\documentclass[12pt,a4paper]{amsart}

\usepackage[utf8]{inputenc}
\usepackage{comment}
\usepackage[T1]{fontenc}
\usepackage{mathtools}
\usepackage[a4paper, margin=2.7cm]{geometry}

\usepackage{amsmath, amsthm, amsfonts, amssymb,bm}
\usepackage{bbm}
\usepackage{mathrsfs}  
\usepackage{enumitem}
\usepackage[colorlinks=true,linkcolor=blue,citecolor=blue,urlcolor=blue,breaklinks]{hyperref}

\usepackage{graphicx}
\usepackage{mathabx}
\usepackage[toc,page]{appendix}
\usepackage{hyperref}
\usepackage{cite}
\usepackage{bigints}
\usepackage[dvipsnames]{xcolor}

\usepackage[toc,page]{appendix}

\newtheorem{theorem}{Theorem}[section]
\newtheorem{proposition}[theorem]{Proposition}
\newtheorem{lemma}[theorem]{Lemma}

\newtheorem{remark}{Remark}[section]

\theoremstyle{definition}
\newtheorem{definition}[theorem]{Definition}

\newcommand{\R}{\mathbb{R}}
\newcommand{\C}{\mathbb{C}}
\newcommand{\NL}{\mathcal{N}}
\newcommand{\N}{\mathbb{N}}

\newcommand{\p}{\partial}
\newcommand{\D}{\mathcal D}
\newcommand{\DD}{\mathcal D}
\newcommand{\dd}{\, \mathrm d}
\newcommand{\LL}{\mathcal{L}}

\renewcommand{\Re}{\mbox{Re}}
\renewcommand{\Im}{\mbox{Im}}
\newcommand{\re}{\textrm{Re}}

\newcommand{\an}[1]{\left\langle #1 \right\rangle}

\newcommand{\grad}{\bigtriangledown}
\newcommand{\br}[1]{\left( #1 \right)}
\newcommand{\sqd}{\sqrt{1-\Delta}}
\newcommand{\no}[1]{\left\|#1\right\| }

\newcommand{\vertiii}[1]{{\left\vert\kern-0.25ex\left\vert\kern-0.25ex\left\vert #1 
    \right\vert\kern-0.25ex\right\vert\kern-0.25ex\right\vert}}

\title{Large time well posedness for a Dirac--Klein-Gordon system}
\date{}

\begin{document}

\author{Federico Cacciafesta}
\address{Federico Cacciafesta: 
Dipartimento di Matematica, Universit\'a degli studi di Padova, Via Trieste, 63, 35131 Padova PD, Italy}
\email{cacciafe@math.unipd.it}

\author{Anne-Sophie de Suzzoni}
\address{Anne-Sophie de Suzzoni:CMLS, \'Ecole  Polytechnique, CNRS, Universit\'e Paris-Saclay, 91128 PALAISEAU Cedex, France}
\email{anne-sophie.de-suzzoni@polytechnique.edu}

\author{Long Meng}
\address{Dipartimento di Matematica, Universit\'a degli studi di Padova, Via Trieste, 63, 35131 Padova PD, Italy}
\email{meng@math.unipd.it}

\author{J\'er\'emy Sok}
\address{Dipartimento di Matematica, Universit\'a degli studi di Padova, Via Trieste, 63, 35131 Padova PD, Italy}
\email{jeremyvithya.sok@unipd.it}


\maketitle

\begin{abstract}
In this paper we prove well posedness for a system coupling a nonlinear Dirac with a Klein-Gordon equation that represents a toy model for the Helium atom with relativistic corrections: the wave function of the electrons interacts with an electric field generated by a nucleus with a given charge density. One of the main ingredients we need is a new family of Strichartz estimates for time dependent perturbations of the Dirac equation: these represent a result of independent interest.
\end{abstract}

\section{Introduction}

The aim of this paper is to study the well posedness for the following system
\begin{equation}\label{eq:system1}
    \left\{\begin{aligned}
    &i\p_t u +\D_m u +W u +|\an{u,\beta u}|^{\frac{p-1}2}\beta u=0,\qquad  u(t,x):\mathbb{R}_t\times\mathbb{R}_x^3\rightarrow\mathbb{C}^{4}\\
    &\p_{tt} W +W -\Delta W=\chi(x-q(t)),\qquad \qquad\qquad W(t,x):\mathbb{R}_t\times\mathbb{R}_x^3\rightarrow\mathbb{R}\\
    &\qquad u(0,x)=u_0(x)\\
    &\qquad W(0,x)=w_0,\quad \p_t W(0,x)=w_1\\
    \end{aligned}
    \right.
\end{equation}
that couples a {\em Dirac} and a {\em Klein-Gordon} equation: here, 
 $\mathcal{D}_m$, with $m>0$, denotes the massive Dirac operator, that can be represented classically as 
 $$
\mathcal{D}_m=-i\displaystyle\sum_{k=1}^3\alpha_k\partial_k+\beta m=-i(\alpha\cdot\nabla)+\beta m
$$
 where the $4\times 4$ Dirac matrices are given by
\begin{equation}
\beta=\left(\begin{array}{cc}I_2 & 0 \\0& I_2\end{array}\right),\qquad \alpha_k=\left(\begin{array}{cc}0 & \sigma_k \\\sigma_k & 0\end{array}\right)
\end{equation}
where $I_2$ is the $2\times2$ identity matrix 
and $\sigma_k$ for $k=1,2,3$ are the Pauli matrices given by
\begin{equation}
\sigma_1=\left(\begin{array}{cc}0 & 1 \\1 & 0\end{array}\right),\quad
\sigma_2=\left(\begin{array}{cc}0 &-i \\i & 0\end{array}\right),\quad
\sigma_3=\left(\begin{array}{cc}1 & 0\\0 & -1\end{array}\right).
\end{equation}
In what follows, we will take for simplicity $m=1$ and we will denote with $\mathcal{D}=\mathcal{D}_1$. The function $\chi$ is assumed to be real and satisfying suitable conditions that we will state later on. 
The nonlinear term we are considering is classical in this setting, and it is the main example of covariant nonlinearities for the Dirac equation, that is such that the equation is left invariant under Lorentz transforms. Notice that our result could be extended to deal with more general nonlinear terms (see Remark \ref{rknl}); nevertheless, we stress the fact that the nonlinear term does not play the staring role in our paper as indeed we will be able to deal with it with standard arguments once we will have proved suitable Strichartz estimates, and therefore we do not feel we need to strive for the maximal generality.



The aim here is to pursue our investigation on the toy models for the Helium atom with relativistic corrections that the first and second author started in \cite{cacdesnoj}. The map $u$ represents the wave function of the electrons, the map $W$ represents the electric field generated by a nucleus centered in the position $q(t)$ at time $t$ and with a charge density given by $\chi(x-q(t))$.  In \cite{cacdesnoj}, the potential generated by the nucleus, instead of being a solution to the Klein-Gordon equation, was an ``electrostatic" approximate field generated by a slowly moving charge, namely, it took the form of a moving Coulomb potential
\[
\frac1{|x-q(t)|}
\]
where $x$ is the space variable and $q(t)$ is the position of the nucleus. Unfortunately, the Coulomb potential in this setting is an approximation that does not help getting a global solution to the Dirac equation. Indeed, the behavior at $0$ is singular, and there is not enough decay at $\infty$ to observe dispersion. The behavior at $0$ is due to the fact that we assume that the nucleus is punctual, and this is an approximation due to the relative smallness of the nucleus with respect to the whole atom. We change the punctuality into a function $\chi$ that represents the charge density of the nucleus (instead of a Dirac delta), and we keep track of the different norms in $\chi$ involved in the proof. The behavior at $\infty$ is due to the fact that we considered electrostatic instead of electrodynamics solutions: by taking electrodynamics solutions (namely solutions to the Klein-Gordon equation, which makes the equations closer to a covariant system), we get indeed a spatial behavior at $\infty$ in $\frac1{r}$ but some time decay due to the dispersion of the Klein-Gordon equation. Note that to be perfectly consistent with the Physics literature, the Klein-Gordon equation should be replaced by the wave equation; however, the wave equation admits less dispersion than the Klein-Gordon one, which prevents us from performing the analysis we present here. To our best knowledge, there is no result about the Dirac wave equation system in dimension $3$ that suits our problem; we mention at least \cite{julsab} where the wave and Klein-Gordon equations are coupled to a Dirac sea. In \cite{julsab}, the extension to global solutions and the uniqueness are proved thanks to Gronwall estimates, which do not apply here.

Regarding the non-linearity, complying with the physics literature would require to include it as a quadratic source term $Q(u,u)$ in the Klein-Gordon equation. Indeed, $u$ generates a a field itself driven by the Klein-Gordon (or wave) equation. The often used non-linearity
\[
(|x|^{-1}*\an{u,\beta u})\beta u 
\]
in the Dirac equation is actually an electrostatic (as opposed to electrodynamics) approximation of a Dirac-wave system with a quadratic non-linearity. However, coupling a Dirac-Klein-Gordon system with a nucleus presents technical difficulties. Indeed, the Dirac-Klein-Gordon system with quadratic non-linearity in the Klein-Gordon equation has been successfully studied by Candy and Herr in \cite{canherr} and uses a technology, namely the $V^2$ spaces introduced by Koch and Tataru in \cite{kochtataru} that does not apply directly to our situation since we have to consider a Dirac equation perturbed by a moving potential. Here, we propose to study a non-linearity that we can deal with standard Strichartz estimates, that is $|\an{u,\beta u}|^{(p-1)/2} \beta u$. We mention the fact that the case $p=3$ (the cubic non-linearity) might be reachable by the same type of argument but with extra loss of derivatives: the issue comes from the absence of a mixed Strichartz-local smoothing estimate for the endpoint in the proof of Theorem \ref{th:3.1} in Subsection \ref{subsec:StrichartzEstimates}. There exists however a mixed Strichartz-local smoothing estimate with extra loss of derivatives, see for instance \cite{cacdan} that would help solving the problem for the cubic non-linearity. 

Our first main result is the following (we postpone the overview of the notation to the end of the introduction):

\begin{theorem}\label{teo1}
Let $p$ and $s\leq2$ be such that:
\begin{itemize}
    \item $s\geq\frac32-\frac1{p-1}$ if $p>3$ is an odd integer,
    \item $(p-1)/2> s\geq \frac32-\frac1{p-1}$ if $p>3$ is not an odd integer.
\end{itemize}
Let $q(t)$ satisfy the following:
\begin{equation}\label{ass-q}
\|\ddot{q}\|_{L^1}\leq \frac12,\qquad \ q\in L^\infty .
\end{equation}
Then, provided , $\no{w_0}_{W^{s+3,1}}$, $\no{w_1}_{W^{s+2,1}}$, $\no{\chi}_{W^{2+s,1}}$,  $\no{\an{x}^{3+}\chi}_{L^\infty }$ and $\no{u_0}_{H^{s}}$ are sufficiently small, system \eqref{eq:system1} is globally well posed in the space 
\begin{equation}\label{Xspace}
    X=C^0_TH^s\cap L^{p-1}_T L^\infty.
\end{equation}
\end{theorem}

\begin{remark}
The choice of the constant $1/2$ in the assumption \eqref{ass-q} is arbitrary, and it could be substituted by any constant $C<1$. We prefer to keep $1/2$ in order not to have a constant to keep track of during our computations.
\end{remark}

\begin{remark}
The assumptions we take on the charge density $\chi$ are technical and driven by our proof; anyway, we stress that at least they include regular functions with compact support. However, they are not entirely satisfying because they ask for the total charge of the nucleus, that is the $L^1$ norm of $\chi$, to be small enough. The condition on the $W^{2+s,1}$ norm may be replaced by a condition on a $W^{\sigma,p}$ with $p>1$ by interpolation. This can be tracked down in the proof in Lemmas \ref{lem:3.3'}, \ref{lem:W3}, \ref{contW3} where we used more dispersion than we actually needed. However, the condition on $\|\an{x}^4\chi\|_{L^\infty}$ cannot be easily lifted. Perhaps an elaborate scaling argument could replace this condition with a smallness condition on $\|\an{x}^{1+}\chi\|_{L^\infty}$ but this is not obvious to us. 
\end{remark}

\begin{remark}\label{rknl} 
As mentioned, our result could be extended to more general non-linearities: since in what follows we will be quite sketchy when dealing with the nonlinear term, let us here briefly review the theory of the well-posedness for system
\begin{equation*}
\begin{cases}
    i\p_t u +\D u +\NL(u)=0,\qquad  u(t,x):\mathbb{R}_t\times\mathbb{R}_x^3\rightarrow\mathbb{C}^{4}\\
    u(0,x)=u_0(x).
    \end{cases}
\end{equation*} 
The analysis in the case $\NL(u)=|\an{u,\beta u}|^{\frac{p-1}2}\beta u$, $p\geq3$, with the "classical" use of dispersive tools has been started in \cite{escobedo1997semilinear} and subsequently sharpened in a number of works (see e.g. \cite{machihara2003small}, \cite{machetc}, \cite{bournavcandy}, \cite{bejherr}). The use of standard contraction argument based on Strichartz estimates allows to prove well posedness on the space $X$ as given by \eqref{Xspace} in the case $p>3$; on the other hand, dealing with the case $p=3$ (the cubic Dirac equation) turns out to be significantly more delicate, as indeed it forces to work at the level of the endpoint Strichartz estimate, which is known to fail in $3D$, and thus it requires some additional tools. Therefore, we here decided to restrict to the case $p>3$ in Thorem \ref{teo1}; nevertheless, we mention that it might still be possible to include it within the range of our admissible choices, but this would require some additional technicalities (see Remark \ref{rkotherstrichartz}). 
Notice also that the assumptions on the ranges of $p$ and $s$ are necessary in order to ensure that the nonlinear term has the regularity needed to perform nonlinear estimates (cfr. \cite{escobedo1997semilinear}).

Another fairly natural choice for the nonlinear term $\NL(u)$ is given by
$$
\NL=(V*|\langle \beta u,u\rangle|^{(p-1)/2})u
$$
with $p\geq3$ and where $V(x)$ is a function such that $|V(x)|\leq |x|^{-\gamma}$ for every $x$ and for some $\gamma>0$ (with suitable limitations). The analysis in this case, still based on the very same strategy on Strichartz estimates (the only difference will be given by the nonlinear estimates) has been performed e.g. in \cite{machtsut}. In this case, the space in which to prove well posedness is in the form
\begin{equation*}
    X=C^0_T H^s\cap L^{p-1}B_{r,2}^k
\end{equation*}
where with $B^s_{p,q}$ we are denoting the Besov spaces (we refer to \cite{bergh} for definitions and basic properties), with a certain range of $s,r,k$. Our proof could be adapted to deal with this nonlinear term as well: this would require to provide estimates on Besov spaces (but this could be easily achieved by making use of the existing ones for the free case), and the rest of our proof would work with minor modifications.
\end{remark}

The proof of Theorem \ref{teo1} is quite standard {\em provided} one has suitable Strichartz estimates at disposal; to the best of our knowledge, they are not available in the form we need, and we thus need to prove them. To begin with, let us give the following
\begin{definition}[Dirac admissible triple]\label{diracpair}
The triple $(p,r,s)$ is \textit{Dirac admissible} if and only if
\[
\|e^{it\DD}f\|_{L^pL^r}\lesssim \|f\|_{H^{s}}.
\]
\end{definition}
\begin{remark}
The standard choice of Dirac admissible triple is the non-endpoint Schr\"odinger admissible triple \cite{danfan}:
\[
\frac{2}{p}+\frac{3}{r}=\frac{3}{2},\quad 2< p\leq \infty,\quad 2\leq r\leq 6,\quad  s=\frac{1}{2}+\frac{1}{p}-\frac{1}{r}.
\]
Actually, to deal with the nonlinear term in system \eqref{eq:system1}, it is helpful to work with a different triple, that is the one given by 
\begin{equation*}
\left(p-1, \infty, \frac{3}{2}-\frac{1}{p-1}\right),\qquad p>3;
\end{equation*}
 in fact, the estimates in this case can be retrieved by the classical ones and the application of a Gagliardo-Nirenberg inequality (see \cite{escobedo1997semilinear}, Theorem 1.5).
\end{remark}

We thus prove the following
\begin{theorem}[Strichartz estimates]\label{th:3.1}
Let $T\in(0,\infty]$. Let $u=S_V(t)u_0$ be a solution to 
 \begin{equation}\label{diracpot}
\begin{cases}
i\partial_tu +\D u + V(t,x)u=0,\qquad  u(t,x):(-T,T)\times\mathbb{R}_x^3\rightarrow\mathbb{C}^{4}\\
u(0,x)=u_0(x).
\end{cases}
\end{equation}
where $V(t,x)$ is an operator. Let $N>\frac{3}{2}$ and $s\geq 0$. Assume that 
\begin{itemize}
    \item system \eqref{diracpot} is well-posed on $H^s$,
    \item there is a constant $\varepsilon>0$ small enough such that 
 \begin{equation}\label{crucialcondition}
 \|V\|_{T,s,N}:=
\no{\an{x}^N (1-\Delta)^{s/2} V (1-\Delta)^{-s/2}\an{x}^N}_{L^\infty ((-T,T),L^2\rightarrow L^2)}\leq \varepsilon.
\end{equation}
\end{itemize}
Then the following estimate holds:
\begin{equation}\label{strich2}
\no{S_V(t)u_0}_{L^\infty((-T,T) ;H^s)}\lesssim \no{u_0}_{H^{s}}.
\end{equation}
Furthermore, if $(p,r,s)$ is any Dirac admissible triple then the following Strichartz estimates hold:
\begin{equation}\label{strich}
\no{S_V(t)u_0}_{L^p((-T,T) ; L^r)}\lesssim \no{u_0}_{H^{s}}.
\end{equation}
\end{theorem}

\begin{remark}
Strichartz (and more in general dispersive) estimates for potential perturbations of the Dirac equation have been widely investigated (see e.g. \cite{danfan}, \cite{bousdanfan}, \cite{cac1},\cite{cacser}). Anyway, to the best of our knowledge, Theorem \ref{th:3.1} improves on existing results because here $V$ is a time dependent {\em operator}, not necessarily a multiplication one.

\end{remark}

As a second step, we couple system \eqref{eq:system1} with a nuclear dynamics of Hellman-Feynman type, that is we now consider the  following more involved system:
\begin{equation}\label{eq:system2}
    \left\{\begin{aligned}
    &i\p_t u +\D u +W u +|\an{u,\beta u}|^{\frac{p-1}2}\beta u=0;\qquad  u(t,x):\mathbb{R}_t\times\mathbb{R}_x^3\rightarrow\mathbb{C}^{4}\\
    &\p_{tt} W +W -\Delta W=\chi(x-q(t));\qquad  W(t,x):\mathbb{R}_t\times\mathbb{R}_x^3\rightarrow\mathbb{R}\\
    &M\ddot{q}=\an{u,\frac{x-q}{|x-q|^3}u}_{L^2 \C^4}=\int_{\R^3}\an{u(x),u(x)}_{\C^4}\frac{x-q}{|x-q|^3}dx;\\
    &\qquad u(0,x)=u_0(x);\\
    &\qquad W(0,x)=w_0,\quad \p_t W(0,x)=w_1;\\
    &\qquad q(0,x)=0,\quad \dot{q}(0,x)=v_0. 
    \end{aligned}
    \right.
\end{equation}
for some $M\gg 1$ and with the same notations as for system \eqref{eq:system1}. 

This coupling comes from the fact that the electrons act on the nucleus via a potential 
\[
\an{u,\frac{1}{|x-q|}u}_{L^2 \C^4}.
\]
We keep the electrostatic approximation here because the nucleus is far heavier ($M\gg 1$) than the electrons and thus carries some inertia. Hence we assume that its dynamics is driven by the classical dynamics of a charged particle in a given field. 

Note that this type of system has been studied in the nonrelativistic case by Cancès and Le Bris, with electrostatic approximations for the nucleus and the electrons, in \cite{canleb} where the authors prove global well-posedness for the system. We remark that for a nonrelativistic system, the Coulomb potential is not scaling-critical, which makes all the difference with the problem at stake. 

For the system \eqref{eq:system2}, we prove the following:
\begin{theorem}\label{teo2}
Let $p$ and $s\leq 2$ be such that:
\begin{itemize}
    \item $s>\frac32$ if $p>3$ is an odd integer,
    \item $(p-1)/2> s> \frac32$ if $p>3$ is not an odd integer.
    \end{itemize}
  
 Let $\chi$, $w_0,w_1$, $q_1,q_2$ be as in the assumptions of Theorem \ref{teo1} with the additional assumption that $\|\an{x}^{3+}\nabla \chi\|_{L^\infty}$ be sufficiently small.
For all $R>0$, such that $\no{u_0}_{H^s}+\no{w_0}_{W^{s+3,1}}+\no{w_1}_{W^{s+2,1}}+\no{\chi}_{W^{s+1,1}} \leq R$, there exists a constant $C_2=C_2(R)$ such that system \eqref{eq:system2} admits a unique solution $(u,q)$ in the space
\[
C^0([0,T], H^s(\R^3))\times C^2([0,T],\R^3)
\]
for $T\leq C_2 \min(\sqrt M, |v_0|^{-1})$.  
\end{theorem}

\begin{remark} We remark that this result is a large time result in the sense that $M\gg 1$, $|v_0|\ll 1$ and thus $T\gg 1$. The fact that $|v_0|\ll 1$ is due to the fact that the nucleus moves slowly. One expects $|v_0| \lesssim M^{-1}$. However, $\sqrt M$ is not an optimal timescale. The reason why we have a large time and not a global result is due to the fact that the nucleus must not go too far from its initial position. This requires to bound the $L^\infty_T$ norm of its position $q$. We should expect the $L^1_T$ norm of $\frac{\ddot q}{M}$ to be bounded with no restriction on time, since we expect $\frac{u}{|x|}$ to be bounded in $L^2_{t,x}$. However, for technical reasons (in particular the restriction of the regularity of the solution), we were not able to prove such a bound and therefore we get a weaker result. We mention nevertheless that the restriction on the regularity of the solution may be lifted by losing some more derivatives on $\chi$ and this would get us a timescale of $M^{1-\eta}$ for any $\eta>0$.
\end{remark}

\begin{remark}\label{rk1.7}
The regularity assumption $s>3/2$ on the initial condition $u_0$ is needed in order to prove well-posedness for the dynamics of the nuclei or, more precisely, to prove that the map $F(q) = \an{u|\frac{x-q}{|x-q|^3}|v}$ is Lipschitz continuous and thus to be able to apply Picard fixed point Theorem; therefore, it represents an unavoidable threshold. This fact has been already noticed and discussed in \cite{cacdesnoj} (see Remark 1.5 there). On the other hand, the (additional) upper bound $s\leq 2$ turns out to be necessary in view of providing suitable estimates on the function $W$ (see e.g. Lemma \ref{W1properties}). This upper bound is thus due to technical reasons; again, this condition could be lifted at the price of losing derivatives on $\chi$, see remark \ref{rem:IPP}. 
\end{remark}

\medskip

Let us give a brief overview of the structure of the paper: in Section \ref{secstrich1} we will prove Strichartz estimates for solutions to equation \eqref{diracpot} under suitable assumptions on the potential $V$, that is Theorem \ref{th:3.1}. We will rely on the well established path
$$
{\rm virial\: identity}\Rightarrow {\rm weak\: dispersive\: estimates}\Rightarrow {\rm Strichartz \:estimates}.
$$
Section \ref{secmainresults} will be devoted to the proofs of Theorem \ref{teo1}-\ref{teo2}, which will be based on contraction arguments. This will require, as a first step, a careful analysis of the solution to the Klein-Gordon equation in systems \eqref{eq:system1}-\eqref{eq:system2}: we will show indeed that these solutions can be in fact decomposed into the sum of some "dispersive" terms, which then enjoy their own dispersive estimates, plus some "non-dispersive" ones, which on the other hand satisfy condition \eqref{crucialcondition}, so that they allow to recover Strichartz estimates for the perturbed Dirac flow via Theorem \ref{th:3.1}. Once Strichartz estimates are available, the proof of Theorem \ref{teo1} becomes straightforward. For proving Theorem \ref{teo2} on the other hand, we also need to handle the classical dynamics on $q$: to show that it is well posed, we need to assume sufficient regularity on the initial condition $u_0$.

\medskip

{\bf Notation.} We use the standard notation $L^p$ for Lebesgue spaces, often distinguishing with a subscript $x$ (resp. $t$) the norm in space on $\R^3_x$ (resp. in time on $\R_t$); with the subscript $X_T$  we will denote norms on a time interval $(-T,T)$ with $T\in(0,\infty]$, that is e.g. $L^p_T=L^p_t((-T,T))$.  We will denote with $W^{s,p}$ the Sobolev spaces defined as 
\[
\|f\|_{W^{s,p}}:=\left(\sum_{|\alpha|\leq s}\|D^\alpha f\|_{L^p}^p\right)^{1/p},
\]
for $s\in \mathbb{N}$ and $p\ge 1$, and for $s\in(0,\infty)\setminus\mathbb{N}$, let $s=m+r$ with $m\in\mathbb{N}$ and $r\in (0,1)$, then
\[
\|f\|_{W^{s,p}}:=\left(\|f\|_{W^{m,p}}^p+\sum_{|\alpha|=m}\iint_{\mathbb{R}^3\times\mathbb{R}^3}\frac{|D^\alpha f(x)-D^\alpha f(y)|^p}{|x-y|^{4rp}}\dd x\dd y\right).
\]
We will denote with $H^{s,p}$ the spaces equipped with the norms
$$
\|f\|_{H^{s,p}}:=\| H^{s/2}f\|_{L^p}
$$
where $H=\sqrt{1-\Delta}$, for $s\geq 0$ and $p\geq1$, with the usual convention for the case $p=2$ that is $H^s=H^{s,2}$.  According to the interpolation theory, 
\begin{align}\label{eq:H-W}
    \|f\|_{H^{s,p}}\lesssim \|f\|_{W^{s,p}},\qquad 1\leq p,
\end{align}
and according to Calder\'on-Zygmund inequality,
\begin{align}\label{eq:W-H}
    \|f\|_{W^{s,p}}\lesssim \|f\|_{H^{s,p}},\qquad 1<p<\infty.
\end{align}

The Strichartz norms will be denoted as $$\|f\|_{XY}=\|f\|_{X_tY_x}=\|f\|_{X(\R_t;Y(\C^4_x))}$$ where $X$ and $Y$ might be Lebesgue, Sobolev or weighted Sobolev spaces; the local-in-time versions will be written as $X_TY_x=X((-T,T) ;Y(\C^4_x))$ for some $T<\infty$. As declared, we will often omit the subscripts $t$ and $x$ when the context will make it unambiguous.  

We will make use of the following weighted norms: by $L^2(\an{x}^N)$ and $H^1(\an{x}^N)$ we denote respectively the spaces induced by the norms
 \begin{equation}\label{wnorms}
 \no{u}_{L^2(\an{x}^{N})} := \no{\an{x}^N u}_{L^2},\qquad
 \no{u}_{H^1(\an{x}^N)} : = \no{u}_{L^2(\an{x}^N) } + \no{\grad u }_{L^2(\an{x}^N) } 
 \end{equation}
 where $N$ is a real number (that may be negative). Notice that the $H^1(\an{x}^N)$ norm of $u$ is equivalent to the $H^1$ norm of $\an{x}^N u$, which in turns makes it equivalent to the $L^2$ norm of $\D\an{x}^N u$. 

We recall that the norm that will play the staring role, as defined in \eqref{crucialcondition}, is given by
\begin{equation*}
\|V\|_{T,s,N}:=
\no{\an{x}^N (1-\Delta)^{s/2} V (1-\Delta)^{-s/2}\an{x}^N}_{L^\infty ((-T,T) ,L^2\rightarrow L^2)}
\end{equation*}
for $s,N\in \R$. When $T=\infty$, we denote it as $\|V\|_{s,N}$.  
 
 \medskip

{\bf Acknowledgments.} F.C., L.M and J.S. acknowledge support from the University of Padova STARS project ``Linear and Nonlinear Problems for the Dirac Equation" (LANPDE). AS. dS. is supported by the ANR project ESSED ANR-18-CE40-0028.

\section{Linear estimates for the Dirac equation: proof of Theorem \ref{th:3.1} }\label{secstrich1}


This section is devoted to the proof of Theorem \ref{th:3.1}, that is of Strichartz estimates for solutions to equation \eqref{diracpot} under suitable assumptions on the potential $V$. The strategy is  classical in this framework, and it is based on virial identity.

\subsection{Preliminaries}

Let us start with the following equivalence of norms
%
%


\begin{proposition} 
The norm
 \[
 \no{u}_{\tilde H^1(\an{x}^N)} := \no{\D u}_{L^2(\an{x}^N)}+  C\no{u}_{L^2(\an{x}^N)}
 \]
 is equivalent to the $H^1(\an{x}^N)$ one defined in \eqref{wnorms} for $C$ large enough.
\end{proposition}

\begin{proof}

By definition of the Dirac operator, we have that
\[
\no{\D u}_{L^2(\an{x}^N)}^2 = M^2 \|u\|_{L^2(\an{x}^N)}^2 +  A + B
\]
with 
\[
 A = \no{\an{x}^N \alpha_j\partial_j u}_{L^2}^2 \textrm{ and } B = 2\Re \an{-i \an{x}^N \alpha_j \partial_j u , \an{x}^N \beta M u}_{L^2}.
\]
Then, as
\[
 A = 2N \an{x\an{x}^{N-2} u ,\an{x}^N \nabla u} + \|\an{x}^N \nabla u\|_{L^2}^2
\]
 and 
 \[
 B = -i 2N\an{\an{x}^{N-2} x_j \alpha_j  ,M \beta  \an{x}^N u},
 \]
we get 
\[
\no{\D u}_{L^2(\an{x}^N)}^2 \geq \no{u}_{H^1(\an{x}^N}^2 - \br{1 -M^2 + 4N + 2N M}  \no{u}_{L^2(\an{x}^N)}^2
\]
 from which we deduce the result.
 \end{proof}

\begin{proposition}\label{prop:2.6}
 For all $N$ in $\R$, $s>0$, $\alpha\in \mathbb{\mathbb{N}}^3$ and for all $u \in L^2(\an{x}^N)$, if $|\alpha|<s$ the following inequality holds
 \[
  \no{\partial^\alpha H^{-s} u}_{L^2(\an{x}^N)} \lesssim \no{u}_{L^2(\an{x}^N)}.
 \]
 \end{proposition}
 
\begin{proof}
We prove the statement for $N\in\N$, the rest of the cases will be covered by standard interpolation.
By Plancherel theorem, we know that 
\begin{align*}
    \MoveEqLeft\no{\an{x}^N\partial^\alpha H^{-s} u}_{L^2}=\no{H^N \xi^\alpha\an{\xi}^{-s} \widehat{u}}_{L^2}\lesssim \sum_{|\alpha|\leq N}\no{D^\alpha \xi^\alpha \an{\xi}^{-s}\widehat{u}}_{L^2}\\
    \lesssim& \sum_{k=0}^N\no{\an{\xi}^{|\alpha|-k-s}\DD^{N-k}\widehat{u}}_{L^2}\leq\sum_{k=0}^N\no{\an{x}^{N-k}u}_{L^2}\\
    \lesssim &\no{\an{x}^N u}_{L^2}
\end{align*}
and this concludes the proof.
\end{proof}

Let us now introduce some operators that will prove to be very useful in the sequel. Let $v\in\R^3$ with $0<|v|<1$, we define the operator $L_v:\R^3\rightarrow \R^3$ as
\[
L_vx:=\frac{1}{\sqrt{1-|v|^2}}\frac{v\cdot x}{v\cdot v}v+\big(x-\frac{v\cdot x}{v\cdot v}v\big)=\frac{1}{\sqrt{1-|v|^2}}P_v x+P_v^\perp x.
\]
This operator is clearly invertible, and 
\[
L_v^{-1}x:=\sqrt{1-|v|^2}\frac{v\cdot x}{v\cdot v}v+\big(x-\frac{v\cdot x}{v\cdot v}v\big)=\sqrt{1-|v|^2}P_v x+P_v^\perp x.
\]
In particular when $v=0$ we define $L_0 x=L_0^{-1}x=x$. Based on $L_v$, we also define the operator $\LL_v$ and its inverse as follows:
\[
\LL_vf(x)=f(L_vx),\quad \LL^{-1}_vf(x)=f(L^{-1}_vx).
\]
Notice that 
\[
\LL_v(fg)=(\LL_v f)(\LL_vg).
\]
Finally, we define the operators
\[
(-\Delta_v)^{s/2}=\LL_v(-\Delta)^{s/2} \LL^{-1}_v=\left(\sqrt{1-|v|^2}\left|\frac{v\cdot \nabla}{v\cdot v}v\right|^2+\left|\nabla-\frac{v\cdot \nabla}{v\cdot v}\cdot v\right|^2\right)^{s/2}
\]
and 
\begin{align}\label{H_v}
    H_v^s=\LL_vH_v^s\LL_v^{-1}=(1-\Delta_v)^{s/2}.
\end{align}
It is not difficult to see that
\begin{align}\label{eq:equi-oper1}
    \|(-\Delta_v)^{s/2}u\|_{L^2}\lesssim \|(-\Delta)^{s/2}u\|_{L^2}\lesssim \|(-\Delta_v)^{s/2}u\|_{L^2}
\end{align}
and
\begin{align}\label{eq:equi-oper2}
    \|H_v^{s}u\|_{L^2}\lesssim \|H^{s}u\|_{L^2}\lesssim \|H_v^{s}u\|_{L^2}.
\end{align}
Letting $y=L_v^{-1} x$ we get
\begin{align*}
    \MoveEqLeft \|\LL_v^{-1}f\|_{L^p}^p=\int|\LL_v^{-1}f(x)|^p\dd x=\int|\LL_v^{-1}f(L_v y)|^p  \dd L_v y \\
    \nonumber
    =&\int|f(y)|^p  \dd L_v y =\frac{1}{\sqrt{1-|v|^2}}\|f\|_{L^p}^p.
\end{align*}
Thus, for any $1\leq p\leq \infty$, we have
\begin{align}\label{eq:equiv}
    \|f\|_{L^p}\lesssim\|\LL_v^{-1}f\|_{L^p}\lesssim \|f\|_{L^p}.
\end{align}

\subsection{Weak dispersive estimates}

Now, we prove a weak dispersive estimate for solutions to \eqref{diracpot}, that is the following

\begin{proposition}\label{prop:3.2}
Let $T\in(0,\infty]$, $N>\frac{3}{2}$ and $s\geq0$. Assume that $V \in  C((-T,T), H^{s} \rightarrow H^{s})$. There exists a constant $\varepsilon>0$ small enough such that assuming 
    \begin{equation}\label{abb}
   \|V\|_{T,s,N}\leq \varepsilon,
    \end{equation}
then the following estimate holds
\begin{equation}\label{lsest}
 \|u\|_{L^2_TH^s(\an{x}^{-N})} \leq C(\varepsilon) \|u_0\|_{H^s}
\end{equation}
for some constant $C(\varepsilon)$ dependent on $\varepsilon$.
\end{proposition}
\begin{remark}
Notice that this result in particular implies 
\begin{equation}\label{freesmooth}
\no{S_0(t)u_0}_{L^2_TL^2(\an{x}^{-N})}\lesssim \no{u_0}_{L^2}
\end{equation}
for any $N>\frac32$, as indeed condition \eqref{abb} is obviously satisfied when $V=0$.
\end{remark}

\begin{remark} Assuming $V \in  C((-T,T), H^{s+1} \rightarrow H^{s+1})\cap  C^1((-T,T), H^{s-1}\rightarrow H^{s-1})$ ensures enough propagation of regularity in order for the computations below to make sense. Indeed, taking $\tilde V$ as below, we get propagation of regularity at the $H^s$ level, and we can make sense of $\partial_t (\tilde V v)$ which is sufficient to conclude. However, we can pass to the limit in the resulting estimate, first in the initial datum, then in the operator $V$ to get the result of Proposition \ref{prop:3.2}. 
\end{remark}

\begin{proof}[Proof of Proposition \ref{prop:3.2}]
First of all, we introduce the function $v=H^{s-1}u$ that satisfies the equation \begin{equation}\label{v-eq}
i\partial_t v=\D v+\widetilde{V}v
\end{equation}
with $\widetilde{V}=H^{s-1}VH^{1-s}$ and $v_0=H^{s-1}u_0$.
The advantage of using the function $v$ is in that we now aim to prove an estimate at the $H^1$ level on it (which in fact is the "natural setting" for the weak dispersive estimates with our strategy), and the $H^1$ norm of $v$ is equivalent to the $H^s$ norm of $u$.

As it is often the case when dealing with Dirac equation, in order to build a useful virial identity we consider the squared system, that is
\[
-\partial_t^2 v=\D^2v+\D \widetilde{V} v+i\partial_t(\widetilde{V} v).
\]
Let $\psi$ be some real, regular function to be chosen later; we then set
\[
\Theta:=2\Re \an{[-\Delta,\psi]v,\partial_t v}+2\Re\an{[-\Delta,\psi]v,i \widetilde{V} v},
\]
so that
\[
\partial_t \Theta=2\Re\an{[-\Delta,\psi]v,\partial_t^2 v+i\partial_t (\widetilde{V}v)}+2\Re\an{[-\Delta,\psi]\partial_t v, i\widetilde{V}v}
\]
Using \eqref{v-eq}, we get
\[
\partial_t\Theta=-2\Re\an{[-\Delta,\psi]v, \D^2 v}+A+B
\]
with
\[
A=2\Re\an{[-\Delta,\psi]\partial_t v,i \widetilde{V}v}
\]
and
\[
B=-2\Re\an{[-\Delta,\psi]v, \D \widetilde{V}v}.
\]

\begin{lemma}[Right hand side I]
We have
\begin{align*}
\|A\|_{L^1_T}\lesssim& \|\psi\|_2 \|\widetilde{V}\|_{T,1,N}\no{v}_{L^2_T H^1(\an{x}^{-N})}^2\\
\|B\|_{L^1_T}\lesssim&\|\psi\|_2 \|\widetilde{V}\|_{T,1,N}\no{v}_{L^2_T H^1(\an{x}^{-N})}^2.
\end{align*}
where $\|\psi\|_2=\|\nabla \psi\|_{L^\infty _T}+\|\Delta \psi\|_{L^\infty_T }$.
\end{lemma}
\begin{proof}
\textbf{Term A.} We have
\[
A=2\Re\an{[-\Delta,\psi]\partial_t v, i\widetilde{V}v}.
\]
Recalling that $[-\Delta,\psi]=-\Delta \psi-2\nabla\psi\cdot\nabla$ is skew-symmetric, we get
\[
|A|\leq \no{\partial_t v}_{L^2(\an{x}^{-N})}\no{[-\Delta,\psi]\widetilde{V}v}_{L^2(\an{x}^N)}\lesssim\no{\psi\|_2\|\partial_t v}_{L^2(\an{x}^{-N})}\no{\widetilde{V}v}_{H^1(\an{x}^N)}.
\]
Obviously,
\[
\no{\widetilde{V}v}_{H^1(\an{x}^N)}\lesssim \no{\an{x}^N H \widetilde{V} H^{-1}\an{x}^N}_{L^2\rightarrow L^2}\no{v}_{H^1(\an{x}^{-N})}.
\]
We now control $\no{\partial_t v}_{L^2(\an{x}^{-N})}$. From the equation on $v$, we get
\[
\no{\partial_t v}_{L^2(\an{x}^{-N})}=\no{\D v+\widetilde{V}v}_{L^2(\an{x}^{-N})}\leq \no{\D v}_{L^2(\an{x}^{-N})}+\no{\widetilde{V}v}_{H^1}
\]
from which we get
\[
\no{\partial_t v}_{L^2(\an{x}^{-N})}\leq \no{v}_{H^1(\an{x}^{-N})}+\no{\an{x}^N H \widetilde{V} H^{-1}\an{x}^N}_{L^2\rightarrow L^2}\no{v}_{H^1(\an{x}^{-N})}
\]
Using the fact that $\widetilde{V}$ should be small, we get
\[
|A|\lesssim \no{\psi}_2 \no{\an{x}^N H \widetilde{V} H^{-1}\an{x}^N}_{L^2\rightarrow L^2}\no{v}_{H^1(\an{x}^{-N})}^2
\]
The Cauchy-Schwarz inequality on the integral on time gives the result for $A$.

\textbf{Term B.} We have 
\[
B=-2\Re\an{[-\Delta,\psi]v, \D \widetilde{V}v}.
\]
This gives by Cauchy-Schwarz inequality,
\[
|B|\lesssim \no{\psi}_2\no{v}_{H^1(\an{x}^{-N})}\no{\D \widetilde{V} v}
\]
As we have seen previously,
\begin{equation}\label{eq:HVH}
    \no{\D \widetilde{V}v}_{L^2(\an{x}^N)}\lesssim \no{\an{x}^N H \widetilde{V} H^{-1}\an{x}^N}_{L^2\rightarrow L^2}\no{v}_{H^1(\an{x}^{-N})}.
\end{equation}
Using Cauchy-Schwarz inequality on the integral on time, we get the result.
\end{proof}

\begin{lemma}[Left hand side I]\label{lem:3.2}
We have 
\[
\no{\Theta}_{L^\infty_T }\lesssim \no{\psi}_2\no{v}_{L^\infty_T H^1}^2
\]
\end{lemma}
\begin{proof}
Recall that 
\[
\Theta=2\Re \an{[-\Delta,\psi]v,\partial_t \widetilde{V}}+2\Re\an{[-\Delta,\psi]v,i \widetilde{V}v}.
\]
We hence have by H\"older inequality
\[
\no{\Theta}_{L^\infty_T }\lesssim\no{[-\Delta,\psi]v}_{L^\infty_T L^2}\br{\no{\partial_t v}_{L^\infty_T L^2}+\no{\widetilde{V}v}_{L^\infty_T L^2}}.
\]
We have on the one hand
\[
\no{[\Delta,\psi]v}_{L^\infty_T  L^2}\lesssim \no{\psi}_2\no{v}_{L^\infty_T H^1},
\]
and on the other hand, by boundedness of $\widetilde{V}$,
\[
\no{\widetilde{V}v}_{L^\infty_T  L^2}\lesssim \no{v}_{L^\infty_T  L^2}.
\]
Finally, since $i\partial_t v= \D v+\widetilde{V}v$,
\[
\no{\partial_t v}_{L^\infty_T L^2}\lesssim \no{v}_{L^\infty_T H^1}.
\]
This concludes the proof.
\end{proof}

\begin{lemma}[Left hand side II]\label{lem:3.3}
We have that
\[
\no{v}_{L^\infty_T H^1}^2\lesssim \no{v_0}_{H^1}^2+ \|\widetilde{V}\|_{T,1,N}\no{v}_{L^2_TH^1(\an{x}^{-N})}^2
\]
\end{lemma}
\begin{proof}
We proceed as usual. Recall that the equation is well-posed in any $H^s$ with propagation of regularity hence the computation below make sense. We differentiate 
\[
\no{v(t)}_{H^1}^2=\an{H v, H v}
\]
to get
\[
\partial_t \no{v(t)}^2_{H^1}=2\Re\an{H v,H\partial_t v}=2\Im\an{H v,H i\partial_t v}.
\]
Using the equation on $v$,
\[
\partial_t\no{v(t)}_{H^1}^2=2\Im \an{H v, H \widetilde{V} v},
\]
from which we get
\[
\no{v(t)}_{L^\infty_T H^1}^2\leq \no{v_0}_{H^1}+2\no{\an{H v, H \widetilde{V}v}}_{L^1_T}.
\]
Finally, by the inequality (\ref{eq:HVH}) we get the result. 
\end{proof}

\begin{lemma}[Right hand side II]
We have
\[
2\Re\an{[-\Delta,\psi]v, \D^2 v}=\an{\Delta^2 \psi v,v}-4\an{\partial_k v, \partial_k\partial_j\psi \partial_j v}.
\]
\end{lemma}
\begin{proof} The proof is classical, but we include it anyway for completeness's sake. 

As $\D^2 = M^2 - \Delta$, we have
\[
2 \Re \an{[-\Delta ,\psi]v,\D^2 v} = -2 \Re \an{[-\Delta ,\psi]v,\Delta v}.
\]
We recall that $[-\Delta ,\psi]$ is skew-symmetric, and $\Delta$ is self-adjoint. Then we get
\[
-2 \Re \an{[-\Delta ,\psi]v,\D^2 v}  = - \an{[-\Delta ,\psi]v,\Delta v} - \an{\Delta v, [-\Delta ,\psi]v} = \an{ [\Delta, [\Delta, \psi]]v,v}.
\]
We have $[\Delta ,\psi] = \Delta \psi +  2 \nabla \psi \nabla$, which gives
\[
[\Delta,[\Delta,\psi]] = \Delta^2 \psi + 4 \nabla \Delta \psi \cdot \nabla + 4\nabla \otimes \nabla \psi \cdot \nabla \otimes \nabla.
\]
We compute
\[
a := \an{\nabla \otimes \nabla \psi \cdot \nabla \otimes \nabla v , v } = \an{\partial_j \partial_k \psi \partial_j \partial_k v , v}.
\]
We use that $\psi$ is real to get
\[
a = \an{\partial_j \partial_k v, (\partial_j \partial_k \psi) v}.
\]
We use that $\partial_j $ is skew-symmetric and the Leibniz rule to get
\[
a = - \an{\partial_k v , \partial_j^2  \partial_k \psi v } - \an{\partial_k v , \partial_j \partial_k \psi \partial_j v}.
\]
In other words
\[
a= -\an{\nabla \Delta \psi\cdot \nabla v, v}  - \an{\partial_k v , \partial_j \partial_k \psi \partial_j v}.
\]
Summing up, we get
\[
2 \Re \an{[-\Delta ,\psi]v,\D^2 v} = \an{\Delta^2 \psi v, v} - 4 \an{\partial_k v , \partial_j \partial_k \psi \partial_j v}.
\]
\end{proof}

Let us now introduce the multiplier $\psi$, which is completely standard in this contest (see e.g. \cite{bousdanfan}).

\begin{definition}
For all $R>0$ we define the radial function $\psi_R$ such that $\psi_R(0)=0$ and 
\[
\psi'_R(r)=\left\{
\begin{array}{cc}
\frac{3r}{2\an{R}}\quad&\textrm{if }r\leq R\\
\frac{R}{\an{R}}\br{\frac{3}{2}-\frac{1}{2}\frac{R^2}{r^2}}\quad&\textrm{if }r> R.
\end{array}
\right.
\]
\end{definition}
The choice of the multiplier above yields the following properties
\begin{lemma}
We have
\begin{eqnarray}
&\Delta\psi_R=\frac{3}{\an{R}}\mathbbm{1}_{r\leq R}+\frac{R}{\an{R}}\frac{3}{r}\mathbbm{1}_{r>R}\\
&\Delta^2\psi_R=-\frac{3}{R\an{R}}\delta(r-R)\\
&\partial_k\partial_j\psi_R=\delta_j^k\frac{\psi'_R}{r}+\mathbbm{1}_{r>R}\frac{3R}{2\an{R}}\frac{x_jx_k}{r^2}\br{\frac{R^2}{r^2}-1}\\
&\no{\psi_R}_2\leq \frac{9}{2}
\end{eqnarray}
\end{lemma}
\begin{proof}
Straightforward computations.
\end{proof}

\begin{lemma}\label{lem:3.7}
We have for all $R\geq 0$,
\[
-2\Re \an{[-\Delta, \psi_R] v, \D^2 v } \geq  \frac{3}{R\an R} \int_{S_R} |v|^2 + \frac{4}{\an{R}} \int_{B_R}|\nabla v|^2
\]
where $S_R$ is the sphere of radius $R$ and $B_R$ is the ball of radius $R$.
\end{lemma}

\begin{proof} We have 
\[
 \an{ - \Delta^2 \psi_Rv,v } = \int \frac3{R\an R} \delta(r-R) |v|^2 =  \frac{3}{R\an R} \int_{S_R} |v|^2,
 \]
and 
\[
\an{\partial_k v , \partial_k \partial_j \psi_R \partial_j v} = \int_{B_R} \frac{|\nabla v|^2}{\an R} + \frac{R}{\an R} \int_{B_R^c} \Big[ \frac1{r} \Big( \frac32  - \frac12 \frac{R^2}{r^2} \Big) |\nabla v|^2 + \frac{x_kx_j}{r^3} \frac32 \Big( \frac{R^2}{r^2} - 1 \Big) \overline{\partial_k v} \partial_j v\Big].
\]
Let 
\[
a : = \int_{B_R^c} \frac{x_kx_j}{r^3} \frac32 \Big( \frac{R^2}{r^2} - 1 \Big) \overline{\partial_k v} \partial_j v.
\]
We have 
\[
a = \int_{B_R^c} \frac{\overline{x_k\partial_k v} x_j\partial_j v}{r^3} \frac32 \Big( \frac{R^2}{r^2} - 1 \Big) .
\]
Because $ \frac{R^2}{r^2} - 1$ is negative, we get
$$
a \geq \int_{B_R^c} \frac{|\nabla v|^2}{r} \frac32 \Big( \frac{R^2}{r^2} - 1 \Big) .
$$
We now sum up and get
$$
\an{\partial_k v , \partial_k \partial_j \psi_R \partial_j v} \geq \int_{B_R} \frac{|\nabla v|^2}{\an R} + \frac{R}{\an R} \int_{B_R^c}  \frac1{r} \Big( \frac32 \frac{R^2}{r^2}  - \frac12 \frac{R^2}{r^2} \Big) |\nabla v|^2 .
$$
From the positivity of $\int_{B_R^c}  \frac1{r} \Big( \frac32 \frac{R^2}{r^2}  - \frac12 \frac{R^2}{r^2} \Big) |\nabla v|^2 $, we get
$$
\an{\partial_k v , \partial_k \partial_j \psi_R \partial_j v} \geq \int_{B_R} \frac{|\nabla v|^2}{\an R} 
$$
\end{proof}

\begin{lemma}
We have for any $\alpha>\frac{3}{2}$ and $\beta>\frac{1}{2}$,
\begin{eqnarray}
\no{v}_{L^2_{T,x}(\an{x}^{-\alpha})}^2 &\lesssim &\sup_R \int_{-T}^T \frac{R}{\an R} \int_{S_R} |v|^2 \\
\no{\nabla v}_{L^2_{T,x}(\an{x}^{-\beta})}^2 &\lesssim  & \sup_R  \int_{-T}^T \frac1{\an R} \int_{B_R} |\nabla v|^2.
\end{eqnarray}
\end{lemma}

\begin{proof} Let $w(x)= \int_{-T}^T |v|^2(x)$.  We have 
\[
\no{v}_{L^2_{T,x}(\an{x}^{-\alpha})}^2 = \int \frac{w}{\an{|x|}^{2\alpha}} \dd x= \int \dd r  \an{r}^{-2\alpha} \int_{S_r} w
\]
from which we get
\[
\no{v}_{L^2(\an{x}^{-\alpha})}^2  \leq \int dr \frac{r\an r}{\an{r}^{2\alpha}} \sup_R \frac{1}{R\an R} \int_{S_R} w.
\]
Since $\alpha > \frac32$, we have $2\alpha -2 > 1$ and thus $\frac{r\an r}{\an{r}^{2\alpha}} $ is integrable which gives the first result.

Let $z= \int_{-T}^T |\nabla v|^2$, we have 
\[
\no{\nabla v}_{L^2(\an{x}^{-\beta})}^2 = \int \frac{z}{\an{r}^{2\beta}} = \int dr \frac1{\an{r}^{2\beta}}  \int_{S_r} z.
\]
We write $\frac1{\an{r}^{2\beta}} =  \int_{r}^\infty \beta \frac{\tau}{\an{\tau}^{2\beta+2}}d\tau$. We get
\[
\no{\nabla v}_{L^2(\an{x}^{-\beta})}^2 = \int d\tau \frac{\tau}{\an{\tau}^{2\beta+2}} \int_{B_\tau} z
\]
and thus
\[
\no{\nabla v}_{L^2(\an{x}^{-\beta})}^2 = \int d\tau \frac{\tau}{\an{\tau}^{2\beta+1}} \frac1{\an \tau}\int_{B_\tau} z
\]
from which we deduce
\[
\no{\nabla v}_{L^2(\an{x}^{-\beta}}^2 \leq \int d\tau \frac{\tau}{\an{\tau}^{2\beta+1}} \sup_R \frac1{\an R}\int_{B_R} z.
\]
Since $2\beta > 1$, we get that $\frac{\tau}{\an{\tau}^{2\beta+1}}$ is integrable from which we can conclude.
\end{proof}

\begin{lemma}
Let $N>\frac{3}{2}$. If there is a constant $\varepsilon>0$ small enough such that 
    \[
    \|\widetilde{V}\|_{T,1,N}\leq \varepsilon.
    \]
Then the following estimate holds
\[
 \|v\|_{L^2_TH^1(\an{x}^{-N})} \leq C(\varepsilon) \|v_0\|_{H^1}
\]
with some constant $C(\varepsilon)$ dependent on $\varepsilon$.
\end{lemma}
\begin{proof}

We specialize $\Theta_R$ with our choice of $\psi=\psi_R$. We have 
\begin{equation}\label{eq:3.9}
    \int_{-T}^T \partial_t \Theta_R  = \Theta_R(T)- \Theta_R(-T).
\end{equation}
On the one hand we have 
\begin{equation}\label{eq:3.10}
    \int_{-T}^T \partial_t \Theta_R \geq \int_{-T}^T 2\re \an{[-\Delta,\psi_R]v,\D^2 v} - \no{A}_{L^1_T} - \no{B}_{L^1_T}.
\end{equation}
and on the other hand
\begin{equation}\label{eq:3.11}
    -\Theta_R(T) + \Theta_R(-T) \leq 2\no{\Theta_R}_{L^\infty_T }.
\end{equation}
By Lemma \ref{lem:3.7},
\begin{equation}\label{eq:3.12}
    - \int_{-T}^T 2\re \an{[-\Delta,\psi_R]v,\D^2 v} \geq \int_{-T}^T \Big(3 \frac{R}{\an R}\int_{S_R} |v|^2  + 4 \frac1{\an R} \int_{B_R} |\nabla v|^2\Big)
\end{equation}
and by Lemma \ref{lem:3.2} and Lemma \ref{lem:3.3},
\begin{equation}\label{eq:3.13}
    \no{\Theta}_{L^\infty_T }\lesssim \|\psi_R\|_2 \br{\no{v_0}_{H^1}^2+ \|\widetilde{V}\|_{T,1,N}\no{v}_{L^2_T H^1(\an{x}^{-N})}^2}.
\end{equation}
Given the bounds on $A$ and $B$ and combining the inequality (\ref{eq:3.9})-(\ref{eq:3.13}), we get the existence of $C$ such that
\begin{multline*}
C \no{\psi_R}_2 \no{v_0}_{H^1}^2 \\ \geq \int_{-T}^T \Big(3 \frac{R}{\an R}\int_{S_R} |v|^2  + 4 \frac1{\an R} \int_{B_R} |\nabla v|^2\Big)
- C\no{\psi_R}_2  \|\widetilde{V}\|_{T,s,N}\no{v}_{L^2_T H^1(\an{x}^{-N})}^2.
\end{multline*}
Since $\no{\psi_R}_2$ is uniformly bounded in $R$, we get the existence of $C_2$ such that
\begin{multline*}
C_2  \no{v_0}_{H^1}^2 \geq \int_{-T}^T \Big(3 \frac{R}{\an R}\int_{S_R} |v|^2  + 4 \frac1{\an R} \int_{B_R} |\nabla v|^2\Big)
- C_2 \|\widetilde{V}\|_{T,s,N}\no{v}_{L^2_T H^1(\an{x}^{-N})}^2.
\end{multline*}
Let $0 < \varepsilon \leq 1$, and let us assume that $\|\widetilde{V}\|_{s,N} \leq \varepsilon$. Thus,
\[
\int_{-T}^T \Big(3 \frac{R}{\an R}\int_{S_R} |v|^2  + 4 \frac1{\an R} \int_{B_R} |\nabla v|^2\Big) 
- C_2 \varepsilon \no{v}_{L^2_T H^1(\an{x}^{-N})} ^2
\leq C_2 \no{v_0}_{H^1}^2 .
\]
We take the sup in $R$ and we get
\[
\sup_R \int_{-T}^T \Big(3 \frac{R}{\an R}\int_{S_R} |v|^2  + 4 \frac1{\an R} \int_{B_R} |\nabla v|^2\Big) - C_2 \varepsilon \no{v}_{L^2_T H^1(\an{x}^{-N})} ^2
\leq C_2 \no{v_0}_{H^1}^2.
\]
As $N > \frac32$, we have 
$$
\|v\|_{L^2_T H^1(\an{x}^{-N})} \lesssim \sup_R \int_{-T}^T \Big(3 \frac{R}{\an R}\int_{S_R} |v|^2  + 4 \frac1{\an R} \int_{B_R} |\nabla v|^2\Big) .
$$
And thus there exists $C_3$ such that
$$
\sup_R \int_{-T}^T \Big(3 \frac{R}{\an R}\int_{S_R} |v|^2  + 4 \frac1{\an R} \int_{B_R} |\nabla v|^2\Big) \Big( 1 - \varepsilon C_3\Big)
\leq C_2 \no{v_0}_{H^1}^2.
$$
By taking $\varepsilon $ small enough, we know there exists a constant $C(\varepsilon)$ dependent on $\varepsilon$ such that
\[
 \|v\|_{L^2_T H^1(\an{x}^{-N})} \leq C(\varepsilon) \|v_0\|_{H^1}.
\]
\end{proof}

{\em Conclusion of the proof of Proposition \ref{prop:3.2}} We use the fact that $\widetilde{V}=H^{s-1}VH^{1-s}$ and $v=H^{s-1}u$ to conclude estimate \eqref{lsest}, as indeed
\[
\|V\|_{T,s,N}=\|\widetilde{V}\|_{T,1,N},
\]
and
\[
\no{u}_{L^2_T H^s(\an{x}^{-N})}^2  = \no{v}_{L^2_T H^1(\an{x}^{-N})}^2,\quad \no{u_0}_{H^s} = \no{v_0}_{H^1}.
\]
\end{proof}

\subsection{Strichartz estimates}\label{subsec:StrichartzEstimates}

We are in position for proving Strichartz estimates for solutions to \eqref{diracpot}.



\begin{proof}[Proof of Theorem \ref{th:3.1}]
By Duhamel's formula, we know that
\[
u(t)=S_0(t)u_0-i\int_0^t S_0(t-\tau)V(\tau,\cdot)u(\tau,\cdot)\dd \tau.
\]
We prove \eqref{strich}: we write
\[
\no{u}_{L^p_T L^q}=
\no{S_V(t)u_0}_{L^p_T L^q}\leq \no{S_0(t)u_0}_{L^p_T L^q}+\no{\int_0^t S_0(t-\tau)V(\tau,\cdot)u(\tau,\cdot))\dd \tau}_{L^p_T L^q}.
\]
Thanks to the Christ-Kiselev Lemma in \cite{christ2001maximal}, since we are only interested in the non-endpoint case ($p>2$), it is sufficient to estimate the untruncated integral
\[
\int S_0(t-\tau)V(\tau,\cdot)u(\tau,\cdot)\dd \tau=S_0(t)\int S_0(-\tau)V(\tau,\cdot)u(\tau,\cdot).
\]
As $(p,q,s)$ is Dirac admissible, according to Definition \ref{diracpair} we get
\begin{align*}
\MoveEqLeft\no{S_0(t)\int_0^T S_0(-\tau)V(\tau,\cdot)u(\tau,\cdot)}_{L^p_T L^q}\lesssim\no{\int_0^T S_0(-\tau)V(\tau,\cdot)u(\tau,\cdot))}_{H^s}.
\end{align*}
Now, we use the dual form of estimate \eqref{freesmooth} to obtain
\begin{align*}
\no{\int_0^T S_0(-\tau)V(\tau,\cdot)u(\tau,\cdot))}_{H^s}\leq& \no{\int_0^T S_0(-\tau)H^{s}V(\tau,\cdot)u(\tau,\cdot))}_{L^2}\\
\lesssim& \no{\an{x}^N H^s (V u)}_{L^2_{T,x}}.
\end{align*}
Hence by Proposition \ref{prop:3.2} and the assumption 
\[
\|V\|_{T,s,N}
\leq \varepsilon,
\]
we finally get
\begin{align*}
\MoveEqLeft\no{S_V u_0}_{L^p_T L^q}\lesssim \no{u_0}_{H^s}+\no{\an{x}^N H^s V H^{s}\an{x}^N \an{x}^{-N}H^{s} u}_{L^2_{T,x}}\\
\lesssim& \no{u_0}_{H^{s}}+\|V\|_{T,s,N}\no{u}_{L^2_T H^{s}(\an{x}^{-N})}\\
\lesssim& \no{u_0}_{H^{s}}
\end{align*}
and this concludes the proof of \eqref{strich}. Estimate \eqref{strich2} can be proved in much the same way, using also the fact that $\|S_0(t)u\|_{H^s}=\|u\|_{H^s}$.

\end{proof}

\begin{remark}\label{rkotherstrichartz}
With minor modifications of the argument above, by relying on the estimates proved in \cite{machihara2003small} for the free flow, it would be possible to prove the following Strichartz estimates in Besov spaces in the assumptions of Theorem \ref{th:3.1}:
\begin{equation}\label{strich1}
\no{S_V(t)u_0}_{L^{p}_TB^0_{r,2}}\lesssim \no{u_0}_{H^{s}}
\end{equation}
for
$$\frac{2}{p}=(2+\theta)(\frac{1}{2}-\frac{1}{r}),\quad 2s=(4+\theta)(\frac{1}{2}-\frac{1}{r}),
$$
$$
 0\leq \theta\leq 1,\quad 2< p,r\leq \infty, \quad (p,r)\neq (\infty,2).
 $$
 The case $p=2$ is slightly more delicate, as indeed in this case it is not possible to rely on the Christ-Kiselev Lemma. In order to deal with it, it might be necessary to use the mixed Strichartz-smoothing estimates with angular derivatives proved in \cite{cacdan}; we omit the details on this case.
\end{remark}

We also have some form of continuity in the operator $V$ in the sense of the following proposition.

\begin{proposition}\label{gencont}
Let $(p,r,s)$ be a Dirac admissible triple as given by Definition \ref{diracpair} and $T\in(0,\infty]$. Let $N>3/2$ and let $V_1$, $V_2$ be two operators belonging to $C((-T,T), H^s\rightarrow H^s)$ such that
\[
\|V_j\|_{T,s,N} \ll 1 .
\]
for $j=1,2$. Let $u_0 \in H^s$. Then the following bounds hold: 
\begin{equation}\label{cont1}
\|S_{V_1}(t) u_0 - S_{V_2}(t)u_0\|_{L^p_TL^q} \lesssim \|V_1-V_2\|_{T,s,N} \|u_0\|_{H^s}.
\end{equation}

\begin{equation}\label{cont2}
\|S_{V_1}(t) u_0 - S_{V_2}(t)u_0\|_{L^\infty_T H^s} \lesssim \|V_1-V_2\|_{T,s,N} \|u_0\|_{H^s}.
\end{equation}
\end{proposition}

\begin{proof} 
We prove \eqref{cont1}. Setting $u_j(t) = S_{V_j}(t) u_0$ for $j=1,2$, from Duhamel's formula we get
\[
u_1(t) - u_2(t) = -i\int_{0}^t S_0(t-\tau)V_1(\tau)\Big( u_1(\tau) - u_2(\tau)\Big) d\tau -i\int_{0}^t S_0(t-\tau)\Big( V_1(\tau) - V_2(\tau)\Big) u_2(\tau) d\tau
.
\]

We get by the inequality \eqref{freesmooth}
\[
\|u_1-u_2\|_{L^2_T H^s((\an{x}^{-N}))} \lesssim \|V_1\|_{T,s,N} \|u_1-u_2\|_{L^2_T H^s((\an{x}^{-N}))} + \|V_1-V_2\|_{T,s,N} \|u_2\|_{L^2_T H^s(\an{x}^{-N})}.
\]
Taking $V_1$ and $V_2$ small enough and using local smoothing on $S_{V_2}$, we get
\[
\|u_1-u_2\|_{L^2_T L^2(\an{x}^{-N})} \lesssim  \|V_1-V_2\|_{T,s,N} \|u_0\|_{H^s}.
\]
Finally, using the same strategy as in the previous proof, we get
\[
\|u_1-u_2\|_{L^p_T L^r}\lesssim \|V_1\|_{T,s,N} \|u_1-u_2\|_{L^2_T H^s(\an{x}^{-N})} + \|V_1-V_2\|_{T,s,N} \|u_0\|_{H^s}
\]
and we conclude using the first inequality we proved.
The proof of \eqref{cont2} follows the same lines, the only difference being that we use estimate \eqref{strich2} instead of \eqref{strich}.
\end{proof}

\subsection{More about condition \eqref{crucialcondition}}
As it is clear, condition \eqref{crucialcondition} is needed for the validity of Strichartz estimates. This condition might be in general not easy to check: therefore, 
we now provide a bound in the case in which $V$ is a function, that will turn out to be very useful in the sequel.

\begin{lemma}\label{lem:XHVHX'}
Let $N\ge 0$, $s\in \R$ and $v\in \mathbb{R}^3,$ %
$|v|\leq 1/2$. Then for any function $V:\,\mathbb{R}^3\rightarrow \mathbb{C}$, we have the following bounds:
    \begin{align*}
    \MoveEqLeft\no{\an{x}^NH^s VH^{-s}\an{x}^N }_{L^2\rightarrow L^2}\lesssim \no{H^s_v\an{x}^{2N} V}_{L^\infty }.
\end{align*}
\end{lemma}
\begin{proof}

Up to taking the dual operator, we can assume $s\ge 0$.
We decompose
\[
\an{\cdot}^NH^sVH^{-s}\an{\cdot}^N=\big(\an{\cdot}^NH^s\an{\cdot}^{-N}H^{-s}\big)\,(H^s\an{\cdot}^N V\an{\cdot}^{N}H^{-s}\big)\,\big(H^s\an{\cdot}^{-N}H^{-s}\an{\cdot}^N\big).
\]
Using \eqref{eq:equi-oper2} and then the Kato-Ponce inequality for $H^s_v$ \eqref{K-Phv}, we get
\[
\no{H^s\an{x}^N V\an{x}^{N}H^{-s}}_{L^2\to L^2}\lesssim  \no{H_v^{s} \an{x}^{2N}V}_{L^\infty }+\no{ \an{x}^{2N}V}_{L^\infty }\lesssim \no{H_v^{s} \an{x}^{2N}V}_{L^\infty }.
\]

To conclude the proof, we need to show that $$F_1:=\an{x}^NH^s\an{x}^{-N}H^{-s}$$ and $$F_2:=H^s\an{x}^{-N}H^{-s}\an{x}^N$$ are bounded $L^2$-operators. 
This technical point is proven in Appendix \ref{sec:f_1_f_2}.
\end{proof}

\section{Proof of Theorems \eqref{teo1}-\eqref{teo2}}\label{secmainresults}

In this section we provide the proofs for Theorems \eqref{teo1} and \eqref{teo2}, that is we prove well posedness for systems \eqref{eq:system1} and \eqref{eq:system2}.

\subsection{The Klein-Gordon equation on \texorpdfstring{$W$}{}}\label{sec:3.1}

We start by studying the Klein-Gordon equation in systems \eqref{eq:system1}-\eqref{eq:system2} that prescribes the dynamics of the potential $W$. The main idea of this subsection is to represent the solution in order to seperate a "dispersive " and a "non dispersive part". It is well known indeed that, by the Duhamel's formula, $W$ can be written as
\begin{equation}\label{eqkg}
W(t,x)=\cos(\sqd t)w_0 +\frac{\sin(\sqd t)}{\sqd }w_1-\int_0^t \frac{\sin(\sqd (t-\tau))}{\sqd }\chi(x-q(\tau))\dd \tau.
\end{equation}

In our case, it is possible to provide a much more explicit  representation of the solution:
\begin{proposition}\label{prop:2.8}
Let $W$ solve the Klein-Gordon equation
$$
\p_{tt} W +W -\Delta W=\chi(x-q(t)).
$$
Then it is possible to decompose $W$ as follows
\[
W(q,\dot q,\ddot q)(t,x)=W_1(q,\dot q)(t,x)(t,x)+W_2(t,x)+W_3(q,\dot q,\ddot q)(t,x)(t,x)
\]
with 
\[
W_1(q,\dot q)(t,x):=\chi_1(\dot q,x-q(t)),
\]
\[W_2(t,x):=\cos(\sqd t)w_0 +\frac{\sin(\sqd t)}{\sqd }w_1+\frac{\cos(\sqd t)}{1-\Delta}\chi(x),
\]
\[
W_3(q,\dot q,\ddot q)(t,x):=\int_0^t\br{e^{i\sqd(t-\tau)}\chi_2(q(\tau),\dot q(\tau),\ddot q(\tau))-e^{-i\sqd(t-\tau)}\chi_3(q(\tau),\dot q(\tau),\ddot q(\tau))}\dd\tau,
\]
where
\[
\begin{aligned}
\widehat{\chi}_1(\dot q,\xi)=&\frac{\widehat{\chi}(\xi)}{\an{\xi}^2+\br{i\xi\cdot \dot{q}(t)}^2},\\
\widehat{\chi}_2(q,\dot q,\ddot q,\xi)=&\frac{\widehat{\chi}(\xi)}{2i\an{\xi}}e^{-i\xi\cdot q(t)}\frac{i\xi\cdot \ddot{q}(t)}{\br{-i\an{\xi}-i\xi\cdot \dot{q}(t)}^2},\\
\widehat{\chi}_3(q,\dot q,\ddot q,\xi)=&\frac{\widehat{\chi}(\xi)}{2i\an{\xi}}e^{-i\xi\cdot q(t)}\frac{i\xi\cdot \ddot{q}(t)}{\br{i\an{\xi}-i\xi\cdot\dot{q}(t)}^2}.
\end{aligned}
\]
\end{proposition}
\begin{proof}
First of all, let 
\[
U=\int_0^t \frac{\sin(\sqd (t-\tau))}{\sqd }\chi(x-q(\tau))\dd \tau.
\]
We pass in Fourier variables to obtain
\[
\begin{aligned}
\widehat{U}=&\int_0^t\frac{\sin(\an{\xi} (t-\tau))}{\an{\xi}}\widehat{\chi}(\xi)e^{-i\xi\cdot q(\tau)}\dd \tau\\
=&\frac{\widehat{\chi}(\xi)}{2i\an{\xi}}e^{i\an{\xi}t}\int_0^t e^{-i\an{\xi}\tau-i\xi\cdot q(\tau)}\dd \tau -\frac{\widehat{\chi}(\xi)}{2i\an{\xi}}e^{-i\an{\xi}t}\int_0^t e^{i\an{\xi}\tau-i\xi\cdot q(\tau)}\dd \tau.
\end{aligned}
\]
Let 
\[
I_\pm=\frac{\widehat{\chi}(\xi)}{2i\an{\xi}}e^{\pm i\an{\xi}t}\int_0^t e^{\mp i\an{\xi}\tau-i\xi\cdot q(\tau)}\dd \tau;
\]
then 
\[
\widehat{U}=I_+-I_-.
\]
Integrating by parts, we get
\[
\begin{aligned}
I_\pm=&\frac{\widehat{\chi}(\xi)}{2i\an{\xi}}e^{\pm i\an{\xi}t}\int_0^t \frac{\mp i\an{\xi}-i\xi\cdot\dot{q}(\tau)}{\mp i\an{\xi}-i\xi\cdot \dot{q}(\tau)} e^{\mp i\an{\xi}\tau-i\xi\cdot q(\tau)}\dd \tau\\
=&\frac{\widehat{\chi}(\xi)}{2i\an{\xi}}e^{\pm i\an{\xi}t}\left(\frac{e^{\mp i\an{\xi}\tau-i\xi\cdot q(\tau)}}{\mp i\an{\xi}-i\xi\cdot \dot{q}(\tau)}\pm\frac{1}{ \an{\xi}}-\int_0^t e^{\mp i\an{\xi}\tau-i\xi\cdot q(\tau)}\frac{-i\xi\cdot \ddot{q}(\tau)}{(\mp i\an{\xi}-\xi\cdot\dot{q}(\tau))^2} \right)\\
=&\frac{\widehat{\chi}(\xi)}{2i\an{\xi}}\left(\frac{e^{-i\xi\cdot q(\tau)}}{\mp i\an{\xi}-i\xi\cdot \dot{q}(\tau)}\pm \frac{e^{\pm i\an{\xi}t}}{i \an{\xi}}-\int_0^t e^{\pm i\an{\xi}(t-\tau)-i\xi\cdot q(\tau)}\frac{-i\xi\cdot \ddot{q}(\tau)}{(\mp i\an{\xi}-\xi\cdot\dot{q}(\tau))^2} \right).
\end{aligned}
\]
Computing $I_+-I_-$, we get
\[
\frac{e^{-i\xi\cdot q(\tau)}}{i\an{\xi}-i\xi\cdot \dot{q}(\tau)}-\frac{e^{-i\xi\cdot q(\tau)}}{- i\an{\xi}-i\xi\cdot \dot{q}(\tau)}=e^{-i\xi\cdot q(\tau)} \frac{2i\an{\xi}}{\an{\xi}^2-(\xi\cdot \dot{q}(\tau))^2}
\]
and
\[
\frac{e^{i\an{\xi}t}}{ i\an{\xi}}+\frac{e^{-i\an{\xi}t}}{ i\an{\xi}}=-2i\frac{\cos(\an{\xi}t)}{\an{\xi}}
\]
from which we get
\begin{align*}
\MoveEqLeft \widehat{U}= \frac{\widehat{\chi}(\xi)e^{-i\xi\cdot q(\tau)}}{\an{\xi}^2+(i\xi\cdot \dot{q}(\tau))^2}-\frac{\widehat{\chi}(\xi)\cos(\an{\xi}t)}{\an{\xi}^2}\\
&+\int_0^t \frac{\widehat{\chi}(\xi)}{2i\an{\xi}}e^{-i\xi\cdot q(\tau)}  \left(\frac{i\xi\cdot \ddot{q}(\tau) e^{i\an{\xi}(t-\tau)}}{(-i\an{\xi}-\xi\cdot\dot{q}(\tau))^2}-\frac{i\xi\cdot \ddot{q}(\tau) e^{-i\an{\xi}(t-\tau)}}{(i\an{\xi}-\xi\cdot\dot{q}(\tau))^2}\right)\dd \tau
\end{align*}
and this concludes the proof.
\end{proof}

Now, in view of applying a contraction argument to prove the well posedness for our differential systems, we need to provide some estimates on the terms $W_j$, $j=1,2,3$. The idea is that to deal with the term $W_1$ we will make use of Theorem \ref{th:3.1}, and thus we will check that the potential $W_1$ satisfies the necessary conditions, while for the terms $W_2$ and $W_3$ we will exploit their own dispersive properties driven by the Klein-Gordon flow.

To begin with, we first of all show that the functions $\chi_1,\chi_2$ and $\chi_3$ can be seen as convolution terms:

\begin{lemma}\label{lem:FT_chi}
	Let $Y(x)=\frac{e^{-|x|}}{4\pi|x|}$, $Z(x)=e^{-|x|}$ and let $K_1$ be the modified Bessel function of the second kind. For any $v\in \mathbb{R}^3$ with $|v|<1$, up to some multiplicative constants we have:
	\begin{enumerate}
		\item $\mathscr{F}\big(\frac{1}{\an{\xi}^2-(\xi\cdot v)^2}\big)=\frac{1}{\sqrt{1-|v|^2}}Y(L_{v} x) \in W^{1,1}(\mathbb{R}^3)$ 
		\item $\mathscr{F}\big(\frac{1}{(\an{\xi}^2-(\xi\cdot v)^2)^2}\big)=\frac{1}{\sqrt{1-|v|^2}}Z(L_{v} x)\in W^{3,1}(\mathbb{R}^3)$,\\
		\item $\mathscr{F}\big(\frac{1}{\an{\xi}}\big)=\frac{K_1(|x|)}{|x|}\in L^1(\mathbb{R}^3)$.
	\end{enumerate}
\end{lemma}
\begin{proof}
To compute the Fourier transforms we use the identity
\begin{equation}\label{eq:ft_radial}
    \int_{\mathbb{R}^3}f(|x|^2)e^{ix\cdot p}dx=\frac{2\pi}{i|p|}\int_{-\infty}^{\infty}rf(r^2)e^{ir|p|}dr.
\end{equation}
By Cauchy's residue formula we easily find that the Fourier transform of $(1+|\xi|^2)^{-1}$ (resp. $(1+|\xi|^2)^{-2}$) is $e^{-|x|}/(4\pi|x|)$ (resp. $Ce^{-|x|}$ for some $C>0$).
Both functions are integrable. Furthermore, $Y(x)\in W^{1,1}$ and $Z(x)\in W^{3,1}$. We get the formula of $\mathscr{F}(\an{\xi}^{-1})$ by showing that $|x|\mathscr{F}(\an{\xi}^{-1})$ satisfies the same ODE as $K_1$'s. 
We recall that $K_1(r)$ has exponential decay and diverges at $r=0$ with singularity $\frac{1}{r}$ \cite[Sections 3.71 and 7.23]{Watson}. Thus $\frac{K_1(|x|)}{|x|}$ is in $L^1$.
	
Now, for any $\xi$, we have the decomposition $\xi=P_v\xi+P_v^\perp \xi$. Setting $z_1=P_v\xi$, and $(z_2,z_3)=P_v^\perp \xi$, we have
\[
\an{\xi}^2-(\xi\cdot v)^2= 1+(1-|v|^2)|P_v \xi|^2+|P_v^\perp \xi|^2.
\]
Changing the variables,  we conclude by the dilation formula for the Fourier transform:
\begin{align*}
\MoveEqLeft \mathscr{F}\left(\frac{1}{\an{\xi}^2-(\xi\cdot v)^2}\right)=\frac{1}{(2\pi)^3}\int_{\mathbb{R}^3}\frac{1}{\an{\xi}^2-(\xi\cdot v)^2} e^{i\xi\cdot x}\dd\xi\\
=&\frac{1}{(2\pi)^3}\int_{\mathbb{R}^3}\frac{1}{1+\sqrt{1-|v|^2}z_1^2+z_2^2+z_3^2} e^{i(z_1\cdot P_v x+(z_2,z_3) \cdot  P_v^\perp x)}\dd z_1 \dd z_2\dd z_3\\
=&\frac{1}{\sqrt{1-|v|^2}}Y(\frac{1}{\sqrt{1-|v|^2}}P_v x+P_v^\perp x).
\end{align*}
We get the Fourier transform of $\frac{1}{(\an{\xi}^2-(\xi\cdot v)^2)^2}$ in a similar fashion.
\end{proof}

We now estimate the terms $W_j$ one by one. 
\begin{lemma}[Estimates on $W_1$]\label{W1properties} 
Let $N\geq 0$, $T\in(0,\infty]$ and $s\in[0,2]$. If $|\dot{q}|\leq \frac{1}{2}$, and $q\in L^\infty _T$,  then
\[
\no{H^s_{\dot q}\an{x}^{2N} W_1}_{L_T^\infty L^\infty }\lesssim \an{\no{q}_{L^\infty _T}}^{2N}\no{\an{x}^{2N}\chi}_{L^\infty }.
\]
\end{lemma}

\begin{proof}

If suffices to prove the case $s=2$ and $s=0$, and the conclusion follows from the standard interpolation.

Recall $W_1=H^{-2}_{\dot q}\tau_{q}\chi$ with $\tau_{q}(\chi)(z)=\chi(x-q)$. By definition of $H^s_{\dot q}$ and \eqref{eq:equiv}, it suffices to estimate $\|H^s\an{L_{\dot q}^{-1}x}^{2N}H^{-2}\mathcal{L}_{\dot q}^{-1}\tau_{q}\chi\|_{L^\infty_TL^\infty}$ for $s\in [0,2]$. For the case $s=0$, thanks to Lemma \ref{lem:FT_chi}, we have
\begin{equation}\label{eq:s0}
    \begin{aligned}
    \MoveEqLeft
    \|\an{L_{\dot q}^{-1}x}^{2N}H^{-2}\mathcal{L}_{\dot q}^{-1}\tau_{q}\chi\|_{L^\infty_TL^\infty}\\ &=\no{\an{L_{\dot q}^{-1}x}^{2N}\int_{\mathbb{R}^3}Y(y)\mathcal{L}^{-1}_{\dot q}\tau_{q}\chi(x-y)\dd y}_{L^\infty _T L^\infty }\\
    &\lesssim\int_{\mathbb{R}^3}\no{Y(y)\an{L_{\dot q}^{-1}(x-y)}^{2N}\mathcal{L}^{-1}_{\dot q}\chi(x-y-q(t))}_{L^\infty_T L^\infty }\dd y\\
    &\quad+\int_{\mathbb{R}^3}\no{\an{L_{\dot q}^{-1}(y)}^{2N}Y(y)\mathcal{L}^{-1}_{\dot q}\chi(x-y-q(t))}_{L^\infty_T L^\infty }\dd y\\
    &\lesssim \|\an{x-q}^{2N}\chi\|_{L^\infty }\\
    &\lesssim \an{\|q\|_{L^\infty _T}}^{2N}\|\an{x}^{2N}\chi\|_{L^\infty }
\end{aligned}
\end{equation}
where as $|x|\lesssim |\mathcal{L}^{-1}_{\dot q}x|\lesssim |x|$ the second inequality holds. On the other hand,  we have
\begin{align*}
    \MoveEqLeft H^2\an{L_{\dot q}^{-1}x}^{2N} H^{-2}\mathcal{L}_{\dot q}^{-1}\tau_{q}\chi=(-\Delta \an{L_{\dot q}^{-1}x}^{2N})H^{-2}\mathcal{L}_{\dot q}^{-1}\tau_{q}\chi\\&
    -\nabla \an{L_{\dot q}^{-1}x}^{2N}\cdot \int_{\mathbb{R}^3}\nabla Y(y) \mathcal{L}_{\dot q}^{-1}\tau_{q}\chi(x-y) \dd y+\an{L_{\dot q}^{-1}x}^{2N}\mathcal{L}_{\dot q}^{-1}\tau_{q}\chi.
\end{align*}
Mimicking the estimate \eqref{eq:s0}, and using the exponential decay properties of $\nabla Y$, we get the result. This ends the proof.
\end{proof}

\begin{lemma}\label{lem:contW1} Let $N\in\mathbb{R}^+$, $T\in(0,\infty]$ and $s\in[0,2]$. Let $q_1$ and $q_2$ in $ W^{1,\infty}_T$, and assume that  $\|\dot q_1\|_{L^\infty_T },\|\dot q_2\|_{L^\infty_T }$ are less than $\frac12$.
We have 
\begin{equation}\label{eq:W1-W2}
\begin{aligned}
\MoveEqLeft  
  \|W_1(q_1,\dot q_1)-W_1(q_2,\dot q_2)\|_{T,s,N}\\
  \lesssim& \br{\an{\no{q_1}_{L^\infty_T }}^{2N}+\an{\no{q_2}_{L^\infty_T }}^{2N}}\no{\an{x}^{2N}\nabla \chi}_{L^\infty }\|q_1- q_2\|_{W^{1,\infty}_T }.
\end{aligned}
\end{equation}
\end{lemma}

\begin{proof}
First of all, we observe that
\begin{align*}
    \MoveEqLeft W_1(q_1,\dot q_1)-W_1(q_2,\dot q_2)
    =\int_{0}^1\nabla_{q} W_1(q_1+\tau(q_2-q_1),\dot q_1)(q_2-q_1) \dd\tau\\
    &+ \int_{0}^1\nabla_{\dot q} W_1(q_2,\dot q_1+\tau(\dot q_2-\dot q_1))(\dot q_2-\dot q_1) \dd \tau.
\end{align*}
Thus,
\begin{equation}\label{eq:34}
    \begin{aligned}
    \MoveEqLeft \|W_1(q_1,\dot q_1)-W_1(q_2,\dot q_2)\|_{T,s,N}\\
    &\leq\|q_1-q_2\|_{L^\infty_T } \sup_{\tau\in[0,1]}
    \| \nabla_q W_1(q_1+\tau(q_2-q_1),\dot q_1)\|_{T,s,N}
   \\
    &+\|\dot q_1-\dot q_2\|_{L^\infty_T } \sup_{\tau\in[0,1]}
    \| \nabla_{\dot q} W_1(q_2,\dot q_1+\tau(\dot q_2-\dot q_1))\|_{T,s,N}.
\end{aligned}
\end{equation}
Let $v_\tau(t)=q_1(t)+\tau(q_2(t)-q_1(t))$. Thus, from Lemma \ref{lem:XHVHX'}
\begin{align*}
 \MoveEqLeft   
 \|\nabla_q W_1(v_\tau(t),\dot q_1)\|_{T,s,N}\lesssim \|H^{s}_{v_\tau}\an{x}^{2N}\nabla_x\chi_1(\dot q_1, x-v_\tau(t))\|_{L^\infty_T L^\infty }.
\end{align*}
Notice that 
\begin{align*}
 \MoveEqLeft   \nabla_x\chi_1(\dot q, x-v_\tau(t))=\int_{\mathbb{R}^3}\frac{1}{\sqrt{1-|\dot q_1|^2}}Y(L_{\dot q_1}y) \nabla_x\chi(x-v_\tau(t)-y)\dd y;
\end{align*}
as a consequence of Lemma \ref{W1properties} we get
\begin{align*}
 \MoveEqLeft   
 \| \nabla_q W_1(v_\tau(t),\dot q_1)\|_{T,s,N}\lesssim \sup_{\tau\in[0,1]}\an{\no{v_\tau}_{L^\infty_T }}^{2N}\no{\an{x}^{2N}\nabla \chi}_{L^\infty_T }\\
 \lesssim& \br{\an{\no{q_1}_{L^\infty_T }}^{2N}+\an{\no{q_2}_{L^\infty_T }}^{2N}}\no{\an{x}^{2N}\nabla \chi}_{L^\infty_T }
\end{align*}
Now, we consider the second term in the right hand side of \eqref{eq:34}. Set $w_\tau=\dot q_1(t)+\tau(\dot q_2(t)-\dot q_1(t))$. Obviously, $\no{w_\tau}_{L^\infty_T }\leq (1-\tau)\no{\dot q_1}_{L^\infty_T }+s\no{\dot q_2}_{L^\infty_T }\leq \frac{1}{2}$. From Lemma \ref{lem:XHVHX'},
\begin{align*}
    \MoveEqLeft
    \| \nabla_{\dot q} W_1(q_2,\dot q_1+\tau(\dot q_2-\dot q_1))\|_{T,s,N}\lesssim\no{H^s_{\dot q_2}\an{x}^{2N}\nabla_{\dot q}\chi_1(w_\tau,x-q_2)}_{L^\infty_T L^\infty }.
\end{align*}
Notice that
\begin{align*}
    \MoveEqLeft\nabla_{\dot q}\widehat{\chi}_1(w_\tau,x-q_2)(t,\xi)=\frac{2\xi(\xi\cdot w_\tau)\widehat{\chi}(\xi)e^{-i\xi\cdot q(\tau)}}{(\an{\xi}^2+(i\xi\cdot w_\tau)^2)^2}.
\end{align*}
Let $\widehat{G}_\tau(t,\xi)=\frac{2(\xi\cdot w_\tau)}{(\an{\xi}^2+(i\xi\cdot w_\tau)^2)^2}$; according to Lemma \ref{lem:FT_chi}, we find that $G_\tau(t,x)\in W^{2,1}(\mathbb{R}^3)$. Moreover,
\[
\nabla_{\dot q}\chi(w_\tau,x-q_2)=\int_{\mathbb{R}^3}\nabla_y G_\tau(t,y)\chi(x-q_2-y)\dd y=\int_{\mathbb{R}^3} G_\tau(t,y)\nabla\chi(x-q_2-y)\dd y.
\]
As for the proof of Lemma \ref{W1properties}, we find
\begin{align*}
    \MoveEqLeft \no{\an{x}^{2N}\nabla_{\dot q}\chi(w_\tau,x-q_2)}_{L^\infty L^\infty_T }\lesssim \an{\no{q_2}_{L^\infty_T }}^{2N}\no{\an{x}^{2N}\nabla\chi}_{L^\infty },
\end{align*}
and hence \eqref{eq:W1-W2} follows.

\end{proof}

\begin{lemma}[Estimates on $W_2$]\label{lem:3.3'}
 Let $T\in(0,\infty]$ and $s\in[0,2]$. We have
 \[
    \no{W_2}_{L_T^1 H^{s,\infty}}\lesssim \no{w_0}_{W^{s+3,1}}+\no{w_1}_{W^{s+2,1}}+\no{\chi}_{W^{s+1,1}},
    \]
\end{lemma}
\begin{proof}
It is easy to see that $W_2$ is the solution of the following linear Klein-Gordon equation:
\[
\partial_{tt}W+W-\Delta W=0;\quad W(0,x)=w_0+(1-\Delta)^{-1}\chi(x),\quad \partial_t W(0,x)=w_1.
\]
According to \cite[Corollary 2.3]{hormander1987remarks} and \eqref{eq:H-W}, 
\begin{align}\label{eq:decay}
    \no{W_2(t,\cdot)}_{H^{s,\infty}}\lesssim (1+|t|)^{-3/2}\br{\no{w_0}_{W^{s+3,1}}+\no{w_1}_{W^{s+2,1}}+\no{\chi}_{W^{s+1,1}}}.
\end{align}
The result follows immediately.
\end{proof}

\begin{lemma}[Estimates on $W_3$]\label{lem:W3}
	Let $T\in(0,\infty]$, $s\in[1,2]$ and  $\|\ddot{q}(t)\|_{L^1_T}\leq \frac{1}{2}$.
	Then there exists $C=C(\varepsilon)$ such that:
	\[
	\|W_3\|_{L^1_TH^{s,\infty}}\lesssim\|\chi\|_{W^{2+s,1}}.
	\]
\end{lemma}


\begin{proof}
By symmetry of treatment we only deal with the $\chi_2$-term.
We rewrite:
\begin{align*}
\widehat{\chi}_2(q,\dot q,\ddot q)(t,\xi)&=\frac{\widehat{\chi}(\xi)}{2i\an{\xi}}e^{-i\xi\cdot q(t)}\frac{i\xi\cdot \ddot{q}(t)}{\br{-i\an{\xi}-i\xi\cdot \dot{q}(t)}^2},\\
	&=-\frac{e^{-i\xi\cdot q(t)}}{(\an{\xi}^2-[\xi\cdot\dot{q}(t)]^2)^2}\frac{(\an{\xi}-\xi\cdot\dot{q}(t))^2}{2\an{\xi}}\xi\cdot\ddot{q}(t)\widehat{\chi}(\xi).
\end{align*}
As $|\dot{q}(t)|\leq\|\ddot{q}\|_{L^1_T}\leq \frac{1}{2}$ by assumption, according to Lemma \ref{lem:FT_chi} we have
\begin{equation}\label{eq:chi2}
\begin{aligned}
    \MoveEqLeft \chi_2(q,\dot q,\ddot q)(t,x)=-\frac{1}{2}H^{-1}(H-\nabla\cdot \dot q)^2\nabla\cdot \ddot q \int_{\mathbb{R}^3}\frac{1}{\sqrt{1-|\dot{q}|^2}}Z(L_{\dot{q}} y)\chi(x-y-q(t))\dd y.
\end{aligned}
\end{equation}
Hence Young's convolution inequality gives, for all $s\ge 1$,
\[
	\|H^s\chi_2(q,\dot q)(t,x)\|_{L^1_x}\lesssim |\ddot{q}(t)| \sup_{t\in[0,T]}\|H^{s-1}\chi(x-q(t))\|_{L^1}\lesssim |\ddot{q}| \|\chi\|_{W^{s-1,1}}.
\]
From the decay estimate \eqref{eq:decay} follows that
\begin{align*}
\|H^sW_{3}(q,\dot q,\ddot q)\|_{L_x^\infty}&\le \sum_{j=2,3}\int_0^t\|e^{i(t-\tau)H}H^s\chi_j(q(\tau),\dot q(\tau),\ddot q(\tau))\|_{L_x^\infty}\dd\tau,\\
&\lesssim\sum_{j=2,3}\int_0^t\frac{\|H^s\chi_j(q(\tau),\dot q(\tau),\ddot q(\tau))\|_{W^{3,1}}}{(1+|t-\tau|)^{3/2}}\dd\tau,\\
&\lesssim\|\chi\|_{W^{s+2,1}}\int_0^t\frac{|\ddot{q}(\tau)|}{(1+|t-\tau|)^{3/2}}\dd\tau.
\end{align*}
Integrating in $t$ and using the fact that $(1+|t|)^{-3/2}\in L^1(\mathbb{R})$ we deduce
\begin{align*}
 \MoveEqLeft   \|W_{3}\|_{L^1,H^{s,\infty}}\lesssim \|\chi\|_{W^{s+2,1}}\int_{[0,T]}\int_0^t \frac{|\ddot{q}(\tau)|}{(1+|t-\tau|)^{3/2}}\dd\tau\dd t\\
=&\|\chi\|_{W^{s+2,1}} \int_{[0,T]}\int_{\tau}^T \frac{|\ddot{q}(\tau)|}{(1+|t-\tau|)^{3/2}}\dd t\dd \tau\\
\lesssim& \|\chi\|_{W^{s+2,1}}
\end{align*}
and this concludes the proof.
\end{proof}

\begin{lemma}\label{contW3}
	Let $T\in(0,\infty]$, $s\in[1,2]$, $q_j\in W^{2,1}_T$ and $\|\ddot{q}_j(t)\|_{L^1_T}\leq \frac{1}{2}$ for $j=1,2$.
	Then there exists $C=C(\varepsilon)$ such that: 
	\begin{equation}\label{contwest}
	\|W_3(q_1,\dot q_1,\ddot q_1)-W_3(q_2,\dot q_2,\ddot q_2)\|_{L^1_T H^{s,\infty}}\lesssim \|q_1-q_2\|_{W^{2,1}_T}\no{\chi}_{W^{s,1}}.
	\end{equation}
\end{lemma}
\begin{proof}
We have
\begin{align*}
\MoveEqLeft\|H^s(W_3(q_1,\dot q_1,\ddot q_1)-W_3(q_2,\dot q_2,\ddot q_2))\|_{L_x^\infty}\\
\le& \sum_{j=2,3}\int_0^t\|e^{i(t-t')H}H^s(\chi_j(q_1(t'),\dot q_1(t'),\ddot q_1(t'))-\chi_j(q_2(t'),\dot q_2(t'),\ddot q_2(t')))\|_{L_x^\infty}\dd t',\\
\lesssim&\sum_{j=2,3}\int_0^t\frac{\|H^s(\chi_j(q_1(t'),\dot q_1(t'),\ddot q_1(t'))-\chi_j(q_2(t'),\dot q_2(t'),\ddot q_2(t')))\|_{W^{3,1}}}{(1+|t-t'|)^{3/2}}\dd t'.
\end{align*}
Let $v_1=q_2-q_1,\,v_2=\dot q_2-\dot q_1$ and $v_3=\ddot q_2-\ddot q_1$; we have
\begin{align*}
 \MoveEqLeft   \|W_{3}\|_{L^1_T W^{s,\infty}}\lesssim \sum_{j=2,3}\int_{[-T,T]}\int_0^t\frac{\|H^s(\chi_j(q_1(t'),\dot q_1(t'),\ddot q_1(t'))-\chi_j(q_2(t'),\dot q_2(t'),\ddot q_2(t')))\|_{W^{3,1}}}{(1+|t-t'|)^{3/2}}\dd t'\dd t\\
\lesssim&\sum_{j=2,3} \int_{[-T,T]}\|H^s(\chi_j(q_1(t'),\dot q_1(t'),\ddot q_1(t'))-\chi_j(q_2(t'),\dot q_2(t'),\ddot q_2(t')))\|_{W^{3,1}}\dd t'\\
\lesssim&\sum_{j=2,3}\sup_{\tau\in[0,1]}\int_{[-T,T]}\|H^s\nabla_{q}\chi_j(q_1(t')+\tau v_1,\dot q_1(t'),\ddot q_1(t'))\|_{W^{3,1}}|v_1(t')|\dd t'\\
&+\sum_{j=2,3}\sup_{\tau\in[0,1]}\int_{[-T,T]}\|H^s\nabla_{\dot q}\chi_j(q_1(t'),\dot q_1(t')+\tau v_2(t'),\ddot q_1(t'))\|_{W^{3,1}}|v_2(t')|\dd t'\\
&+\sum_{j=2,3}\sup_{\tau\in[0,1]}\int_{[-T,T]}\|H^s\nabla_{\ddot q}\chi_j(q_1(t'),\dot q_1(t'),\ddot q_1(t')+\tau v_3(t'))\|_{W^{3,1}}|v_3(t')|\dd t'.
\end{align*}
We deal with the $\chi_2$-term, as the other one can be dealt with similarly by symmetry. According to \eqref{eq:chi2} and Lemma \ref{lem:FT_chi}, we get
\begin{align*}
 \MoveEqLeft   \|H^s\nabla_{\dot q}\chi_j(q_1(t'),\dot q_1(t')+\tau v_2(t'),\ddot q_1(t'))\|_{W^{3,1}}\\
 =&\|\nabla_x H^s\nabla_{\dot q}\chi_j(q_1(t'),\dot q_1(t')+\tau v_2(t'),\ddot q_1(t'))\|_{W^{3,1}}\lesssim |\ddot q(t')|\no{\chi}_{W^{s,1}},
\end{align*}
and
\begin{align*}
    \MoveEqLeft \|H^s\nabla_{\dot q}\chi_j(q_1(t'),\dot q_1(t')+\tau v_2(t'),\ddot q_1(t'))\|_{W^{3,1}}\lesssim |\ddot q(t')|\no{\chi}_{W^{s-1,1}}
\end{align*}
as well as
\begin{align*}
    \MoveEqLeft \|H^s\nabla_{\ddot q}\chi_j(q_1(t'),\dot q_1(t'),\ddot q_1(t')+\tau v_3(t'))\|_{W^{3,1}}\lesssim\no{\chi}_{W^{s-1,1}}.
\end{align*}
Hence,
\[
\|W_3(q_1,\dot q_1,\ddot q_1)-W_3(q_2,\dot q_2,\ddot q_2)\|_{L^1_T H^{s,\infty}}\lesssim \br{\|q_1-q_2\|_{L^1_T}+\|\dot q_1-\dot q_2\|_{L^1_T}+\|\ddot q_1-\ddot q_2\|_{L^1_T}}\no{\chi}_{W^{s,1}}
\]
and this concludes the proof
\end{proof}

\begin{remark}\label{rem:IPP} Notice that in all of these proofs, the condition $s\le 2$ is due to the low regularity of the convolution functions. We could take $s$ large but at the price of losing derivatives on $\chi$ as mentioned in Remark\ref{rk1.7}.
\end{remark}

\subsection{Strichartz estimates for the Dirac equation in the coupled system}



We now show that solutions to the following equation
\begin{equation}\label{Diracbadpot}
i\partial_t u+\D u +  W_1 u=0.
\end{equation}
where $ W_1 = \chi_1(t,x-q(t))$ satisfy Strichartz estimates: we prove in fact the following
\begin{proposition}\label{propstri}
Let $T\in(0,\infty]$, $(p,r,s)$ any Dirac admissible triple with $s\in[0,2]$, $u$ be a solution to \eqref{Diracbadpot} with initial condition $u(0,x)=u_0(x)$, $q=q(t)$ be such that
$
\|\ddot{q}(t)\|_{L^1_T}\leq \frac12
$  and $q\in L^\infty_T$
and let $\chi$ be such that 
\[
\no{\an{x}^{3+}\chi}_{L^\infty }<C\varepsilon
\]
for some constant $C$ and $\varepsilon$ small enough.
Then $u$ satisfies the Strichartz estimates \eqref{strich}, \eqref{strich1} (and\eqref{strich2}) for the triple $(p,r,s)$.
\end{proposition}

\begin{proof}
We need to check that the operator $W_1$ satisfies the conditions required in Theorem \ref{th:3.1}. To do that, we perform a change of variables, and consider the function
$v(t,x) = u(t,x+q(t))$ which solves the equation
\begin{equation}\label{Diracbadpot2}
i\partial_t v+\mathcal D v + i\beta \dot q (t) \cdot \grad v + \chi_1(t,x) v=0.
\end{equation}
In our assumption on $\ddot q$, we have that $\|\dot q \|_{L^\infty }\leq 1/2$ and this ensures that $H_1:=\mathcal D  + i\beta \dot q (t) \cdot \grad  + \chi_1(t,x)$ is a uniform (in $t$) perturbation of $\D$. Therefore the $L^2$ norm of $H_1f$ is uniformly in time equivalent to the $H^1$ norm of $f$ and $H_1$, which is symmetric, is also essentially self-adjoint. Notice that, 
$$
\partial_t H_1 = -i \dot q \cdot \grad + \beta \partial_t \chi_1,
$$
and as
$$
\partial_t \hat \chi_1 (\xi) = \frac{\hat \chi(\xi)}{(\an \xi^2+(i\xi\cdot \dot q(t))^2)^2} 2  \xi \cdot \dot q \xi \cdot \ddot q,
$$
we get that $\partial_t H_1 $ belongs to $L^1(\R, H^{s+1}\rightarrow H^s)$, and hence $H_1$ is of bounded variations in time as an operator from $H^{s+1}$ to $H^s$. This means in particular that the equation 
$$
i\partial_t v = H_1 v
$$
is well-posed in $H^s$ for any $s\geq0$ as long as $\ddot q$ and $\widehat{\chi} $ are small in $L^1$ norm: in other words, we have that there exists a constant $C>0$ such that for any solution $v$ of \eqref{Diracbadpot2} with initial condition $v_0$ and for any time $t\in\R$ then 
\[
\| v\|_{H^s}\leq C\|v_0\|_{H^s}.
\] 
Now, we re-change variable to get back to the function $u$: as the translations in time do not alter the $H^s$ norm in space, we get for any solution $u$ to equation \eqref{Diracbadpot} the following bound
\[
\| u\|_{H^s}\leq C\|u_0\|_{H^s}
\] 
and thus \eqref{Diracbadpot} is well-posed in $H^s$ for any $s\geq0$.

The other thing we need to check is that the operator $W_1$ satisfies condition \eqref{crucialcondition}. This follows immediately from Lemma \ref{lem:XHVHX'} and Lemma \ref{W1properties}: for any $N\in \mathbb{R}^+$ and $s\in [0,2]$, there is a constant $C'$ such that
\[
\| W_1\|_{T,s,N}\leq C'\an{\no{q}_{L^\infty _T}}^{2N}\no{\an{x}^{2N}\chi}_{L^\infty }.
\]
Let $\no{\an{x}^{3+}\chi}_{L^\infty }$ sufficiently small such that
\[
C'\an{\no{q}_{L^\infty _T}}^{3+}\no{\an{x}^{3+}\chi}_{L^\infty }<\epsilon.
\]
Then, for $s\in [0,2]$
\[
\no{ W_1}_{T,s,3+}<\epsilon.
\]
Applying Theorem \ref{th:3.1}, the conclusion follows.
\end{proof}

\subsection{Proof of Theorem \ref{teo1}}

We are now in position for proving the global existence of solutions for the nonlinear Dirac equation
\begin{eqnarray}\label{eq:Diraconly}
\begin{cases}
i\p_t u +\D u +W u +\NL(u)=0,\\
u(0,x)=u_0(x).
\end{cases}
\end{eqnarray}

According to Proposition \ref{prop:2.8}, we write $W=W_1+W_2+W_3$; letting $V=W_1$, the above Dirac equation can be rewritten in integral form:
\begin{align}\label{solution}
\MoveEqLeft Au=S_0(t)u_0+i\int_0^t S_0(t-\tau) (Wu)(\tau)\dd\tau+i\int_0^t S_0(t-\tau) \NL(u)(\tau)\dd\tau\\
\nonumber
=&S_0(t)u_0+i\int_0^t S_0(t-\tau) ((W_1+W_2+W_3)u)(\tau)\dd\tau+i\int_0^t S_0(t-\tau) \NL(u)(\tau)\dd\tau\\
\nonumber
=&S_{W_1}(t)u_0+i\int_0^t S_0(t-\tau) ((W_2+W_3)u)(\tau)\dd\tau+i\int_0^t S_0(t-\tau) \NL(u)(\tau)\dd\tau
    \end{align}
    The proof of the well posedness is now very standard, and it consists in the application of the contraction mapping principle on the solution map above on the ball
    \begin{equation}\label{Xnorm}
    X_K=\{\psi\in X\left| \|\psi\|_{X}=\right. \no{\psi}_{L^\infty_T H^s}+\no{\psi}_{L^{p-1}_TL^\infty}\leq K)\}
\end{equation}
where $X={L^\infty_T H^s}\cap{L^{p-1}_TL^\infty}$ and $s\in[ s(p),2]$ with $s(p)=\frac32-\frac1{p-1}$.

 The only additional tool that we need (with respect to the unperturbed case) is given by the following Lemma, that allows to control the terms involving $W_2$ and $W_3$:
\begin{lemma}\label{lemw2w3}
Let 
\[
C_{w,\chi}:=\no{w_0}_{W^{s+3,1}}+\no{w_1}_{W^{s+2,1}}+\no{\chi}_{W^{s+1,1}},
\]
and
\[
C_{q,\chi}:=\no{H^s\chi}_{W^{2,1}}.
\]
Then,
\begin{equation}\label{strichW}
\no{\int_0^tS_0(t-\tau)((W_2+W_3)u)\dd\tau}_{X}\lesssim(C_{w,\chi}+C_{q,\chi})\no{u}_{X}.
\end{equation}
\end{lemma}
\begin{proof}
Thanks to Strichartz estimates for the free flow, the left hand side of \eqref{strichW} can be bounded by the term $\no{(W_2+W_3)u}_{L^1H^{s}}$.
By the Kato-Ponce inequality \eqref{K-Pinh}, Lemma \ref{lem:3.3'} and Lemma \ref{lem:W3}, as $s>1$, we then get
\begin{align*}
    \MoveEqLeft \no{(W_2+W_3)u}_{L^1,H^s}\lesssim \no{W_2+W_3}_{L^1H^{s,\infty}}\no{u}_{L^\infty H^s}\\
    \lesssim& (C_{w,\chi}+C_{q,\chi})\no{u}_{L^\infty H^s}\\
    \lesssim& (C_{w,\chi}+C_{q,\chi})\no{u}_{X}
\end{align*}
and this concludes the proof of the Lemma.
\end{proof}
The rest of the proof is now completely standard (see \cite{escobedo1997semilinear}), and we thus omit it.

\medskip

In what follows we will also need the continuity in $q$ of the solution map. We thus  prove the following

\begin{proposition}\label{prop-flow}
 Let $\chi$, $w_0,w_1$, $q_1,q_2$ be as in the assumptions of Theorem \ref{teo1} with the additional assumption that $\|\an{x}^{3+}\nabla \chi\|_{L^\infty}$ is sufficiently small. Let $T\in(0,\infty)$
 and let $\Psi_q$ denote the global flow associated to system \eqref{eq:Diraconly} with $p$ and $s$ as in the assumptions of Theorem \ref{teo1}, and let $\|u_0\|_{H^s}$ small enough. Then $\Psi_q$ satisfies the following properties:
\begin{equation}\label{contbound}
    \|\Psi_{q_j}(t)u_0\|_{L^\infty_T  H^s}\leq C \|u_0\|_{H^s},\qquad j=1,2, 
\end{equation}
\begin{equation}\label{contpsi}
    \|\Psi_{q_1}(t)u_0-\Psi_{q_2}(t)u_0\|_{X}\leq C \|u_0\|_{H^s}\big(\|q_1-q_2\|_{W^{1,\infty}_T }+\|\ddot{q}_1-\ddot{q}_2\|_{L^1_T }\big)
\end{equation}
where the norm $X$ is given in \eqref{Xnorm} and the constant $C=C(s,w_0,w_1,\chi,\|q_1\|_{L^\infty_T },\|q_2\|_{L^\infty_T })$.
\end{proposition}

\begin{proof}
We only need to prove \eqref{contpsi} as indeed \eqref{contbound} is a consequence of the contraction argument.
Let us take $q_1\neq q_2$; we start from representation \eqref{solution} that we rewrite as
\begin{equation*}
u_j=\Psi_{q_j}(t)u_0 =
S_{W_1^{q_j}}(t)u_0+i\int_0^t S_0(t-\tau) ((W_2^{q_j}+W_3^{q_j})u_j)(\tau)\dd\tau+i\int_0^t S_0(t-\tau) \NL(u_j)(\tau)\dd\tau
\end{equation*}
for $j=1,2$. Then we have that 
\begin{equation*}
    \| \Psi_{q_1}(t)u_0-\Psi_{q_2}(t)u_0\|_{X}\leq
    I+II+III,
\end{equation*}
with
\begin{equation*}
    I=\|S_{W_1^{q_1}}(t)u_0-S_{W_1^{q_2}}(t)u_0\|_{X} \leq C^1_{q_1,q_2,\chi}\| q_1- q_2\|_{W^{1,\infty}_T }
\end{equation*}
thanks to Proposition \ref{gencont} and Lemma  \ref{lem:contW1} (the constant $C_{q_1,q_2,\chi}$ is the one given by \eqref{eq:W1-W2}). Then, 
\begin{eqnarray*}
    II&=& \left\|i\int_0^t S_0(t-\tau) \big[(W_2^{q_1}+W_3^{q_1})u_1-(W_2^{q_2}+W_3^{q_2})u_2\big](\tau)\dd\tau\right\|_{X} 
    \\
    &\leq&
    \| (W_2^{q_1}+W_3^{q_1})u_1-(W_2^{q_2}+W_3^{q_2})u_2\|_{L^1_TH^s}
    \\
    &\leq&
    \| (W_2^{q_1}-W_2^{q_2})u_1\|_{L^1_TH^s}+\|(W_3^{q_1}-W_3^{q_2})u_2\|_{L^1_TH^s}+\|(W_3^{q_1}+W_2^{q_2})(u_1-u_2)\|_{L^1_TH^s}
    \\
    &=&II_A+II_B+II_C.
\end{eqnarray*}
Notice now that $II_A=0$ as indeed the term $W_2$ does not depend on $q$. We estimate the other terms as follows:
\begin{equation*}
\begin{aligned}
\MoveEqLeft II_B\lesssim \| W_3^{q_1}-W_3^{q_2}\|_{L^1_TW^{s,\infty}}\|u_2\|_{L^\infty_T  H^s}\leq C^2_{q_1,q_2,\chi}\|u_0\|_{H^s}\|q_1-q_2\|_{W^{2,1}_T }
\end{aligned}
\end{equation*}
where we have used Lemma \ref{contW3} with the constant given in \eqref{contwest} and estimate \eqref{contbound}, and 
\begin{eqnarray*}
II_C&\leq&
    \|(W_3^{q_1}+W_2^{q_2})(u_1-u_2)\|_{L^1H^s}\\
    &\leq&
   \| W_3^{q_1}+W_2^{q_2}\|_{L^1_tW^{s,\infty}}\|u_1-u_2\|_{L^\infty  H^s}\\&\leq& (C_{w,\chi}+C_{q,\chi})\|u_1-u_2\|_{L^\infty  H^s}
\end{eqnarray*}
where the constants are given in Lemma \ref{lemw2w3}. Finally, writing $\NL(u)=|\an{u,\beta u}|^{\frac{p-1}2}\beta u$, combining free Strichartz with classical nonlinear estimates yields
\begin{equation*}
    III= \left\|i\int_0^t S_0(t-\tau) \big[\NL(u_1)-\NL(u_2)\big](\tau)\dd\tau\right\|_{X}\leq C(\no{u_1}^{p-1}_{X }+\no{u_2}^{p-1}_{X })\|u_1-u_2\|_{X}.
\end{equation*}
As shown in the proof of Theorem \ref{teo1}, for $\no{u_0}$ small enough the solution map $\psi$ is contracting, and thus absorbing the necessary terms on the LHS (notice that $T<\infty$) yields \eqref{contpsi}.
\end{proof}

\subsection{Proof of Theorem \ref{teo2}}

We now deal with the proof of Theorem \ref{teo2}, that is we prove local well posedness for system \eqref{eq:system2}. To do this, we essentially follow the strategy developed in \cite{cacdesnoj} (see also \cite{baudouin}).

First of all, we need to deal with the classical dynamics driven by $q$. Let us consider the following system
\begin{equation}\label{systnuclei}
\begin{cases}
\displaystyle
\ddot q=F(q)=\frac{1}{M}\an{\Psi_q u_0|\frac{x-q}{|x-q|^3}|\Psi_q u_0},\\
q(0)=q_0,
\\
\dot q(0)=v_0.
\end{cases}
\end{equation}

We prove the following

\begin{proposition}\label{partialwp}
Let $s\in\big(\frac32,2\big]$. There exists a constant $C$ such that for all $q_0$, $v_0$ and $u_0\in H^s$ system \eqref{systnuclei} admits a unique solution $C^2([0,T])$ for $T\leq\frac{M}{C\|u_0\|_{H^s}^2}$.
\end{proposition}

\begin{proof}
Let $Z$ be the completion of $\mathcal C^2([0,T]$ induced by the norm
\[
q \mapsto \|q\|_{L^\infty} + \|\ddot q\|_{L^1_T}.
\]

We want to apply a contraction (Picard) argument onto the ball
$$
B=B(T)=\left\{q\in Z: \|\ddot q\|_{L^1_T}\leq \frac12, \|q\|_{L^\infty}\leq 1, q(0)=0, \dot q(0)=v_0\right\}.
$$
We denote with $P$ the solution map, that is the map such that $\ddot P(q)=F(q)$.
First of all, we prove that $B$ is stable under the action of $P$:

\begin{lemma}\label{Pmap}
Let $u_0\in H^1$. There exists a constant $C_1$ such that if $T_1\leq \frac{C_1\sqrt M}{\|u_0\|_{H^1}^2}$ then $P$ maps $B$ in $B$.
\end{lemma}

\begin{proof}
Thanks to Hardy inequality we have 
$$
\| F(q)\|_{L^\infty }\leq \frac{C}M\|\Psi_q(u_0)\|_{H^1}^2
$$
which implies
$$
\|F(q)\|_{L^1_T} \leq \frac{CT}M \|u_0\|_{H^1}^2 .
$$
As a consequence we get 
$$
\|  P(q)\|_{L^\infty }\leq |v_0|T+\frac{CT^2}M\|u_0\|_{H^1}^2
$$
so that choosing $T_1\leq K\frac{\sqrt M}{C\|u_0\|_{H^1}^2}$ and $T\leq K |v_0|^{-1}$ with $K$ small enough, we get
$$
\|F(q)\|_{L^1_{T_1}} \leq \frac12, \quad \|  P(q)\|_{L^\infty }\leq 1
$$
that implies that $P(q)\in B$, and so $P$ maps $B$ in $B$.
\end{proof}

Then, we show that $F$ is uniformly Lipschitz-continuous in $q$, that is the following
\begin{lemma}\label{contlemma}
Let $u_0\in H^s$ for some $s\in\big(\frac32,2\big]$ and let $q_1,q_2\in B$. There exists a constant $C_2>0$ and a time $T_2\leq T_1$ with $T_1$ as in Lemma \ref{Pmap} such that

\begin{equation}\label{keycont}
\|P(q_1)-P(q_2)\|_{C^2([0,T_2])}\leq C_2T_2^2\|u_0\|^{2}_{H^s}\|q_1-q_2\|_{C^2([-T_2,T_2]) }
\end{equation}

\end{lemma}

\begin{proof}
We rewrite the difference
$$
F(q_1)-F(q_2)=\an{ \Psi_{q_1}(u_0)|\frac{x-q_1}{|x-q_1|^3}|\Psi_{q_1}(u_0)}-
\an{ \Psi_{q_2}(u_0)|\frac{x-q_2}{|x-q_2|^3}|\Psi_{q_2}(u_0)}
$$
as follows
$$
F(q_1)-F(q_2)=I+II+II
$$
with
$$
I=\an{ (\Psi_{q_1}(u_0)-\Psi_{q_2}(u_0))|\frac{x-q_1}{|x-q_1|^3}|\Psi_{q_1}(u_0)},
$$
$$
II=\an{ \Psi_{q_1}(u_0)|\left(\frac{x-q_1}{|x-q_1|^3}-\frac{x-q_2}{|x-q_2|^3}\right)|\Psi_{q_1}(u_0)}
$$
$$
III=\an{ \Psi_{q_2}(u_0)|\frac{x-q_2}{|x-q_2|^3}|(\Psi_{q_1}(u_0)-\Psi_{q_2}(u_0))}
$$
and we estimate the three terms one by one. We have

For $I$, we write
\begin{eqnarray*}
|I| &\leq& \int_{\R^3}\frac{|\Psi_{q_1}(u_0)- \Psi_{q_2}(u_0)||\Psi_{q_1}(u_0)|}{|q_1-x|^2}
\\
&=&
\int_{\R^3}\frac{|\Psi_{q_1}(u_0)- \Psi_{q_2}(u_0)|}{|q_1-x|^{s-1}}
\frac{|\Psi_{q_1}(u_0)|}{|q_1-x|^{3-s}}
\\
&\leq& C\left\|\frac{\Psi_{q_1}(u_0)- \Psi_{q_2}(u_0)}{|q_1-x|^{s-1}}\right\|_{L^2}
\left\|\frac{\Psi_{q_1}(u_0)}{|q_1-x|^{3-s}}\right\|_{L^2}
\\
 &\leq& C \|\Psi_{q_1}(u_0)- \Psi_{q_2}(u_0)\|_{H^{s-1}} \| \Psi_{q_1}(u_0)\|_{H^{3-s}}
\end{eqnarray*}
where we have made use of \eqref{hardgen}, and thanks to Proposition \ref{prop-flow} we get (notice that $3-s<s$ since $s>\frac32$)
$$
|I| \leq C \|u_0\|_{H^s}^2\|q_1 -q_2\|_{W^{2,1}_T }
\leq 
C \|u_0\|_{H^s}^2\|q_1-q_2\|_{C^2([-T_2,T_2])}.
$$

The same strategy allows to control the term $III$.
To deal with the term $II$,  We consider the quantity
$$
G(q) = \an{u|\frac{x-q}{|x-q|^3}|v}.
$$
We note that after a change of variable (the translation $y = x-q$), we have
$$
G(q) = \an{u_q| \frac{x}{|x|^3}|v_q}
$$
where $u_q(x) = u(x+q)$. After differentiating in $q$, we get
$$
\grad_q G(q)  =  \an{(\grad u)_q| \frac{x}{|x|^3}| v_q} + \an{ u_q|\frac{x}{|x|^3}|\grad v_q}
$$
from which, by the use of \eqref{hardgen}, we obtain
\begin{align*}
\MoveEqLeft|\grad G(q)| \leq C \|(\grad u)_q\|_{H^{2-s}}\|v_q\|_{H^s} + C\|u_q\|_{H^s}\|(\grad v)_q\|_{H^{2-s}} \\
&= C \|u\|_{H^{3-s}}\|v\|_{H^s} + C\|u\|_{H^s}\|v\|_{H^{3-s}}.
\end{align*}
We thus get
$$
|G(q_1) - G(q_2) |\leq |q_1-q_2| \Big(C \|u\|_{H^{3-s}}\|v\|_{H^s} + C\|u\|_{H^s}\|v\|_{H^{3-s}}\Big).
$$
Thus we obtain, as $s>3/2$,
$$
|II|\leq C \|u_0\|_{H^s}\|u_0\|_{H^{3-s}} \| q_1 -  q_2\|_{L^\infty_T } \leq C \|u_0\|_{H^s}^2\| q_1 -  q_2\|_{L^\infty_T }.
$$
We integrate these bounds twice and get the result.

\begin{remark}
Notice how we have used the fact that $s>3/2$ twice, in the application of inequality \eqref{hardgen}: this therefore turns out to be a necessary condition in our proof above.\end{remark}

\end{proof}

Now, the proof of Proposition \ref{partialwp} follows from the two Lemmas: it is a contraction argument for the map $P$ in $B$ for the topology of $Z$.

\end{proof}

{\em Proof of Theorem \ref{teo2}.} We only need to combine Theorem \ref{teo1}, Proposition \ref{prop-flow} and \ref{partialwp}. Let $u_0\in H^s$, $w_0\in W^{s+3,1}$, $w_1\in W^{s+2,1}$ and $\chi \in W^{s+1,1}$ with $s>3/2$. Let $q\in Z$ be the solution to \eqref{systnuclei} as given in Proposition \ref{partialwp}, with $T\leq T_1 := \frac{\sqrt M}{C\|u_0\|^2_{H^s}}, K|v_0|^{-1}$. Let $u=\Psi_q(t)u_0$ defined in Proposition \ref{prop-flow}. Then, the couple $(u,q)\in C([0,T],H^s)\times Z$ for $T\leq MC$ where the constant $C$ depends on $\no{u_0}_{H^s}$, $\no{w_0}_{W^{s+3,1}}$, $\no{w_1}_{W^{s+2,1}}$, $\no{\chi}_{W^{s+1,1}}$ (follows from the proof of Theorem \ref{teo1}), and it satisfies system \eqref{eq:system2}. The fact that $q$ belongs to $\mathcal C^2$ is due to the fact that $\Psi_q(t)(u_0)$ belongs to $\mathcal C(\R, H^s)$ if $q\in Z$.

\medskip

\appendix
\section{Useful inequalities}

We devote this small appendix to recall some useful (and classical) inequalities and small variations of them that were needed during our proofs. 

First, we recall the following generalized Hardy inequality
\begin{equation}\label{hardgen}
\||x|^{-a}(-\Delta)^{-a/2}\|_{L^p\rightarrow L^p}\leq C
\end{equation}
and by duality,
\begin{equation}\label{hardgen1}
    \|(-\Delta)^{-a/2}|x|^{-a}\|_{L^1\rightarrow L^1}\leq C
\end{equation}
which holds for $a>0$, and any $1<p<\frac{3}{a}$ and $p^{-1}+q^{-1}=1$ (see e.g. \cite{herbst1977spectral}). 

Then, we recall the classical Kato-Ponce inequality:
\begin{lemma}[Kato-Ponce inequality]\cite{grafakos2014kato}
For $r\geq 1$, $s\geq 0$ and $1<p_1,q_1,p_2,q_2\leq \infty$ such that $1/r=1/p_1+1/q_1=1/p_2+1/q_2$, we have
\begin{align}\label{K-Ph}
    \no{(-\Delta)^{s/2} fg}_{L^r}\lesssim \no{f}_{L^{p_1}}\no{(-\Delta)^{s/2} g}_{L^{q_1}}+\no{(-\Delta)^{s/2}f}_{L^{p_2}}\no{g}_{q_2},
\end{align}
and
\begin{align}\label{K-Pinh}
    \no{H^s fg}_{L^r}\lesssim \no{f}_{L^{p_1}}\no{H^s g}_{L^{q_1}}+\no{H^sf}_{L^{p_2}}\no{g}_{q_2},
\end{align}
\end{lemma}
It is possible to prove an analog for estimate \eqref{K-Ph} in the case of the operator $H_v^s$ as defined in \eqref{H_v}.
By replacing $fg$ with $\LL_v^{-1}(fg)$, we get the following 
\[
\|H^s\LL_v^{-1}(fg)\|_{L^r}\lesssim \|H^s \LL_v^{-1}f\|_{L^{p_1} }\|\LL_v^{-1} g\|_{L^{q_1}}+\|\LL_v^{-1}f\|_{L^{p_2} }\|H^s\LL_v^{-1}g\|_{L^{q_2}}.
\]
For any $1\leq p\leq \infty$, 
\[
\|H^s\LL_v^{-1}f\|_{L^p}=\|\LL_v^{-1}H_v^sf\|_{L^p}.
\]
Therefore, we get from \eqref{eq:equiv} the following Kato-Ponce inequality:
\begin{align}\label{K-Phv}
    \no{H_v^s fg}_{L^r}\lesssim \no{f}_{L^{p_1}}\no{H_v^s g}_{L^{q_1}}+\no{H_v^sf}_{L^{p_2}}\no{g}_{q_2},
\end{align}
with $1/r=1/p_1+1/q_1=1/p_2+1/q_2$.

\section{End of the proof of Lemma~\ref{lem:XHVHX'}}\label{sec:f_1_f_2}

This part is devoted to the proof of the boundedness as $L^2$-operators of $$F_1:=\an{x}^NH^s\an{x}^{-N}H^{-s},\quad\text{and}\quad F_2:=H^s\an{x}^{-N}H^{-s}\an{x}^N,\quad s,N\ge 0.$$
Recall $H=\sqrt{1-\Delta}$. 
Before turning to the proof, we introduce some notations: we write
\begin{equation}\label{eq:euclid_on_N}
N=2p+r,\ p\in\mathbb{N},\ r\in[0,2),
\end{equation}
and $m_s(\xi)=(1+|\xi|^2)^{s/2}$. Notice that by induction we have for any multi-index $\alpha$
\begin{equation}
 m_{s,\alpha}(\xi):=\partial^\alpha m_s(\xi)=w_\alpha(\xi)\an{\xi}^{s-|\alpha|},
\end{equation}
where $w_\alpha$ is a smooth bounded function (rational function of $\xi$ and $\an{\xi}$).
\subparagraph{\textbf{Boundedness of $F_1$}.}
By Leibniz rule (in Fourier space) we write $F_1$ as a linear combination of terms
\[
\an{x}^{r}m_{s,\alpha_1}(-i\nabla)\frac{x^{\alpha_2}}{\an{x}^N}H^{-s},\ |\alpha_1|+|\alpha_2|\le 2p.
\]

We write $r=\lfloor r\rfloor+\varepsilon$. Using $\an{x}\le 1+|x|$, we realize that we only need to show the boundedness of terms of type
$\an{x}^\varepsilon m_{s,\alpha_1}(-i\nabla)\frac{x^{\alpha_2}}{\an{x}^N}H^{-s}$ with $|\alpha_1|+|\alpha_2|\le 2p+\lfloor r\rfloor$ (recall \eqref{eq:euclid_on_N}).

We now commute $\an{x}^\varepsilon$ with $m_{s,\alpha_1}(-i\nabla)$. The second term  $m_{s,\alpha_1}(-i\nabla)\frac{x^{\alpha_2}\an{x}^\varepsilon}{\an{x}^N}H^{-s}$ is bounded by Kato-Ponce inequality and the estimate $\no{m_{s,\alpha_1}(-i\nabla)H^{|\alpha_1|-s}}_{L^2\to L^2}<\infty$ .
Let us now deal with the commutator term $\big[\an{\cdot}^\varepsilon,m_{s,\alpha_1}(-i\nabla)\big]\tfrac{x^{\alpha_2}}{\an{x}^N}H^{-s}$: we have
\[
\an{x}^\varepsilon=C_\varepsilon \int_0^{\infty}\frac{du}{u^{1-\varepsilon}}\frac{\an{x}^2}{\an{x}^2+u^2},\ \sup_{\varepsilon\in(0,1)}C_\varepsilon<\infty.
\]
Using $\frac{\an{x}^2}{\an{x}^2+u^2}=1-\frac{u^2}{\an{x}^2+u^2}$, we infer
\[
[\an{x}^\varepsilon,m_{s,\alpha_1}(-i\nabla)]=C_\varepsilon\int_0^{\infty}u^{1+\varepsilon}du[m_{s,\alpha_1}(-i\nabla),\frac{1}{\an{x}^2+u^2}]=:A.
\]
We now use Plancherel and estimate the integral kernel $\widehat{A}(\xi,\eta)$ with the help of Lemma ~\ref{lem:FT_chi}. For a test function $\phi$ we have 
\begin{multline*}
4\pi\int_0^{\infty}u^{1+\varepsilon}du[m_{s,\alpha_1}(\xi),\frac{1}{\an{u}^2-\Delta}]\phi=\int_0^{\infty}u^{1+\varepsilon}du\int_{\eta}d\eta\frac{m_{s,\alpha_1}(\xi)-m_{s,\alpha_1}(\eta)}{|\xi-\eta|}e^{-\an{u}|\xi-\eta|}\phi(\eta),\\
	=\int_\eta d\eta \underbrace{\bigg(\int_0^{\infty}u^{1+\varepsilon}e^{-(\an{u}-1/2)|\xi-\eta|}|\xi-\eta|^{2+\varepsilon}du\bigg)}_{\mathcal{I}(|\xi-\eta|)}\frac{e^{-|\xi-\eta|/2}}{|\xi-\eta|^{3+\varepsilon}}\big[m_{s,\alpha_1}(\xi)-m_{s,\alpha_1}(\eta)\big]\phi(\eta).
\end{multline*}

Using $\an{u}-1/2\ge \tfrac{\sqrt{3}}{2}u$, we get $\sup_{\delta\ge 0}\mathcal{I}(\delta)<\infty$.
Using the Taylor expansion of $m_{s,\alpha_1}(\xi)$ with respect to $\eta$ up to order $\lceil s\rceil -|\alpha_1|$, we get:
\[
\frac{e^{-|\xi-\eta|/2}}{|\xi-\eta|^{3+\varepsilon}}(m_{s,\alpha_1}(\xi)-m_{s,\alpha_1}(\eta))=\frac{e^{-|\xi-\eta|/2}}{|\xi-\eta|^{3+\varepsilon}}\Big[\sum_{k=1}^{\lceil s\rceil-|\alpha_1|-1}d^km_{s,\alpha_1}(\eta)(\xi-\eta)^{k}+R_{s,\alpha_1}(\xi,\eta)\Big],
\]
where $|d^km_{s,\alpha_1}(\eta)(\xi-\eta)^{k}|\lesssim |\xi-\eta|^k\an{\eta}^{s-|\alpha_1|-k}$ and the remainder satisfies $|R_{s,\alpha_1}(\xi,\eta)|\lesssim |\xi-\eta|^{s-|\alpha_1|}$. 
Since the function $|\cdot|^{k-3-\varepsilon}e^{-|\cdot|/2}$ is integrable for any $k\ge 1$, it follows that the operator $\big[\an{\cdot}^\varepsilon,m_{s,\alpha_1}(-i\nabla)\big]H^{-s+|\alpha_1|+1}$ is $\|\cdot\|_{L^2\to L^2}$-bounded. Then by Kato-Ponce inequality the operator $H^{s-|\alpha_1|-1}\frac{x^{\alpha_2}}{\an{x}}H^{-s}$ is $\|\cdot\|_{L^2\to L^2}$-bounded. We have shown the boundedness of
\[
\big[\an{\cdot}^\varepsilon,m_{s,\alpha_1}(-i\nabla)\big]\tfrac{x^{\alpha_2}}{\an{x}^N}H^{-s}=\Big(\big[\an{\cdot}^\varepsilon,m_{s,\alpha_1}(-i\nabla)\big]H^{-s+|\alpha_1|+1}\Big)\Big(H^{s-|\alpha_1|-1}\tfrac{x^{\alpha_2}}{\an{x}^N}H^{-s}\Big).
\]



\subparagraph{\textbf{Boundedness of $F_2$}.}
Let us now write 
$$s=2q+t,\quad q\in\N, \quad t\in[0,2).$$ 
By Leibniz rule $[H^{2q},\an{x}^{-N}]$ is a linear combination of terms of type 
$m_{-N,\alpha_1}(x)\partial^{\alpha_2}$ (we recall $m_{-N,\alpha_1}=\partial^{\alpha_1}(\an{\cdot}^{-N})$). Let $\zeta=t-\lfloor t\rfloor\in[0,1)$, we only need to check the boundedness of $H^\zeta m_{-N,\alpha_1}\partial^{\alpha_2}H^{-s}\an{x}^N$ with $|\alpha_1|+|\alpha_2|\le 2q+\lfloor t\rfloor$. 

As for $F_1$, we commute $H^\zeta$ with $m_{-N,\alpha_1}$. We first deal with the second term $m_{-N,\alpha_1}(x)\partial^{\alpha_2}H^{-2q-\lfloor t\rfloor}\an{x}^N:=T_2$. By duality and by Proposition \ref{prop:2.6} we have
\[
\|T_2\|_{L^2\to L^2}=\|\an{x}^N\partial^{\alpha_2} H^{-2q-\lfloor t\rfloor}m_{-N,\alpha_1}\|_{L^2\to L^2}\lesssim \|\an{x}^Nm_{-N,\alpha_1}\|_{L^\infty}<\infty.
\]


Let us now consider the commutator term $[H^\zeta, m_{-N,\alpha_1}(x)]\partial^{\alpha_2}H^{-s}\an{x}^N$ which we further decompose:
\[
[H^\zeta, m_{-N,\alpha_1}(\cdot)]\partial^{\alpha_2}H^{-s}\an{\cdot}^N=\big([H^\zeta, m_{-N,\alpha_1}(\cdot)]\an{\cdot}^N\big)\, \big(\an{\cdot}^{-N}\partial^{\alpha_2}H^{-s}\an{\cdot}^N\big).
\]
As for $T_2$, the operator $\an{x}^{-N}\partial^{\alpha_2}H^{-s}\an{x}^N$ is bounded.
Then, proceeding as we did for $F_1$, given a test function $\phi$ we have
\begin{multline*}
[H^\zeta, m_{-N,\alpha_1}(\cdot)]\an{\cdot}^N\phi=\\
	\frac{C_\zeta}{4\pi}\int_y dy\bigg(\int_0^{\infty}u^{1+\zeta}du e^{-(\an{u}-1/2)|x-y|}|x-y|^{2+\zeta}\bigg)\frac{m_{-N,\alpha_1}(x)-m_{-N,\alpha_1}(y)}{|x-y|^{4\zeta}}e^{-|x-y|/2}	\an{y}^N\phi(y).
\end{multline*}
By the mean-value theorem we have:
\[
|m_{-N,\alpha_1}(x)-m_{-N,\alpha_1}(y)|\an{y}^N\lesssim |x-y|\big(|\partial^{\alpha_1}\nabla m_{-N}(x)|+|\partial^{\alpha_1}\nabla m_{-N}(y)|\big)\an{y}^N.
\]
We have $\sup_y |\partial^{\alpha_1}\nabla m_{-N}(y)|\an{y}^N<\infty$. Then using $\an{y}^N\lesssim \an{x}^N+\an{x-y}^N$ we get
\[
\frac{e^{-|x-y|/2}}{|x-y|^{2+\zeta}}|\partial^{\alpha_1}\nabla m_{-N}(x)|\an{y}^N \lesssim \frac{e^{-|x-y|/2}}{|x-y|^{2+\zeta}}|\partial^{\alpha_1}\nabla m_{-N}(x)|(\an{x}^N+\an{x-y}^N).
\]
Since $\frac{e^{-|y|/2}}{|y|^{2+\zeta}}[1+|y|^N]$ is integrable, we get that $[H^\zeta, m_{-N,\alpha_1}(x)]\an{\cdot}^N$ is bounded.

\medskip

\bibliographystyle{plain}
\bibliography{reference}

\begin{thebibliography}{10}

\bibitem{baudouin}
Lucie Baudouin.
\newblock Existence and regularity of the solution of a time dependent
  {H}artree-{F}ock equation coupled with a classical nuclear dynamics.
\newblock {\em Rev. Mat. Complut.}, 18(2):285--314, 2005.

\bibitem{bejherr}
Ioan Bejenaru and Sebastian Herr.
\newblock The cubic {D}irac equation: small initial data in {$H^1(\Bbb{R}^3)$}.
\newblock {\em Comm. Math. Phys.}, 335(1):43--82, 2015.

\bibitem{bergh}
J\"oran {Bergh} and J\"orgen {L\"ofstr\"om}.
\newblock {\em {Interpolation spaces. An introduction}}, volume 223.
\newblock Springer, Berlin, 1976.

\bibitem{bournavcandy}
Nikolaos Bournaveas and Timothy Candy.
\newblock Global well-posedness for the massless cubic {D}irac equation.
\newblock {\em Int. Math. Res. Not. IMRN}, (22):6735--6828, 2016.

\bibitem{bousdanfan}
Nabile Boussaid, Piero D'Ancona, and Luca Fanelli.
\newblock Virial identity and weak dispersion for the magnetic {D}irac
  equation.
\newblock {\em J. Math. Pures Appl. (9)}, 95(2):137--150, 2011.

\bibitem{cac1}
Federico Cacciafesta.
\newblock Global small solutions to the critical radial {D}irac equation with
  potential.
\newblock {\em Nonlinear Anal.}, 74(17):6060--6073, 2011.

\bibitem{cacdan}
Federico Cacciafesta and Piero D'Ancona.
\newblock Endpoint estimates and global existence for the nonlinear {D}irac
  equation with potential.
\newblock {\em J. Differential Equations}, 254(5):2233--2260, 2013.

\bibitem{cacdesnoj}
Federico Cacciafesta, Anne-Sophie de~Suzzoni, and Diego Noja.
\newblock A {D}irac field interacting with point nuclear dynamics.
\newblock {\em Math. Ann.}, 376(3-4):1261--1301, 2020.

\bibitem{cacser}
Federico Cacciafesta and \'{E}ric S\'{e}r\'{e}.
\newblock Local smoothing estimates for the massless {D}irac-{C}oulomb equation
  in 2 and 3 dimensions.
\newblock {\em J. Funct. Anal.}, 271(8):2339--2358, 2016.

\bibitem{canleb}
Eric Canc\`es and Claude Le~Bris.
\newblock On the time-dependent {H}artree-{F}ock equations coupled with a
  classical nuclear dynamics.
\newblock {\em Math. Models Methods Appl. Sci.}, 9(7):963--990, 1999.

\bibitem{canherr}
Timothy Candy and Sebastian Herr.
\newblock Conditional large initial data scattering results for the
  {D}irac-{K}lein-{G}ordon system.
\newblock {\em Forum Math. Sigma}, 6:Paper No. e9, 55, 2018.

\bibitem{christ2001maximal}
Michael Christ and Alexander Kiselev.
\newblock Maximal functions associated to filtrations.
\newblock {\em J. Funct. Anal.}, 179(2):409--425, 2001.

\bibitem{danfan}
Piero D'Ancona and Luca Fanelli.
\newblock Strichartz and smoothing estimates of dispersive equations with
  magnetic potentials.
\newblock {\em Comm. Partial Differential Equations}, 33(4-6):1082--1112, 2008.

\bibitem{escobedo1997semilinear}
M.~Escobedo and L.~Vega.
\newblock A semilinear {D}irac equation in {$H^s({\bf R}^3)$} for {$s>1$}.
\newblock {\em SIAM J. Math. Anal.}, 28(2):338--362, 1997.

\bibitem{grafakos2014kato}
Loukas Grafakos and Seungly Oh.
\newblock The {K}ato-{P}once inequality.
\newblock {\em Comm. Partial Differential Equations}, 39(6):1128--1157, 2014.

\bibitem{herbst1977spectral}
Ira~W. Herbst.
\newblock Spectral theory of the operator
  {$(p\sp{2}+m\sp{2})\sp{1/2}-Ze\sp{2}/r$}.
\newblock {\em Comm. Math. Phys.}, 53(3):285--294, 1977.

\bibitem{hormander1987remarks}
Lars H\"{o}rmander.
\newblock Remarks on the {K}lein-{G}ordon equation.
\newblock In {\em Journ\'{e}es ``\'{E}quations aux deriv\'{e}es partielles''
  ({S}aint {J}ean de {M}onts, 1987)}, pages Exp. No. I, 9. \'{E}cole Polytech.,
  Palaiseau, 1987.

\bibitem{kochtataru}
Herbert Koch and Daniel Tataru.
\newblock Dispersive estimates for principally normal pseudodifferential
  operators.
\newblock {\em Comm. Pure Appl. Math.}, 58(2):217--284, 2005.

\bibitem{machetc}
Shuji Machihara, Makoto Nakamura, Kenji Nakanishi, and Tohru Ozawa.
\newblock Endpoint {S}trichartz estimates and global solutions for the
  nonlinear {D}irac equation.
\newblock {\em J. Funct. Anal.}, 219(1):1--20, 2005.

\bibitem{machihara2003small}
Shuji Machihara, Kenji Nakanishi, and Tohru Ozawa.
\newblock Small global solutions and the nonrelativistic limit for the
  nonlinear {D}irac equation.
\newblock {\em Rev. Mat. Iberoamericana}, 19(1):179--194, 2003.

\bibitem{machtsut}
Shuji Machihara and Kimitoshi Tsutaya.
\newblock Scattering theory for the {D}irac equation with a non-local term.
\newblock {\em Proc. Roy. Soc. Edinburgh Sect. A}, 139(4):867--878, 2009.

\bibitem{julsab}
Julien Sabin.
\newblock Global well-posedness for a nonlinear wave equation coupled to the
  {D}irac sea.
\newblock {\em Appl. Math. Res. Express. AMRX}, (2):312--331, 2014.

\bibitem{Watson}
G.~N. Watson.
\newblock {\em A treatise on the theory of {B}essel functions}.
\newblock Cambridge Mathematical Library. Cambridge University Press,
  Cambridge, 1995.
\newblock Reprint of the second (1944) edition.

\end{thebibliography}

\end{document}